\documentclass[11pt, a4paper,reqno]{amsart}
\pdfoutput=1
\usepackage{tikz,enumitem,rotating,latexsym,bm,stmaryrd,caption,}
\usetikzlibrary{positioning,intersections,decorations}
\usetikzlibrary{shapes}

\definecolor{ao(english)}{rgb}{0.0, 0.5, 0.0}

\definecolor{darkgreen}{rgb}{0.0, 0.5, 0.0}

\usetikzlibrary{decorations.pathmorphing}

	\definecolor{eng}{rgb}{0.0, 0.5, 0.0}
\definecolor{apple}{rgb}{0.55, 0.71, 0.0}
\definecolor{cadmium}{rgb}{0.0, 0.42, 0.24}
\definecolor{darkspringgreen}{rgb}{0.09, 0.45, 0.27}
\definecolor{amethyst}{rgb}{0.6, 0.4, 0.8}
\definecolor{ao}{rgb}{0.0, 0.0, 1.0}
\definecolor{atomictangerine}{rgb}{1.0, 0.6, 0.4}
\definecolor{carmine}{rgb}{0.59, 0.0, 0.09}

\definecolor{toggle}{rgb}{1.0, 0.94, 0.96}

\usepackage[normalem]{ulem}
 
 \newcommand{\pDelta}{\Delta}

 \newcommand{\csigma}{{{{\color{magenta}\sigma}}}}
 \newcommand{\ctau}{{{{\color{cyan}\tau}}}}

 \newcommand{\al}{{{{\color{magenta} \sigma}}}}
 \newcommand{\crho}{{ \color{darkgreen}\rho}}
  \newcommand{\gam}{{{{\color{ao(english)} \rho}}}}
 \newcommand{\bet}{{{\color{cyan} \tau}}}

\newcommand{\eps}{ \varepsilon}

\newcommand{\w}{{\underline{w}}}
\newcommand{\uvee}{{\underline{v} }}
\newcommand{\vvv}{{\underline{z} }}
\newcommand{\y}{{\underline{y}}}
\newcommand{\x}{{\underline{x}}}

\newcommand{\grade}{v}

\usepackage{standalone}

\usepackage{etex}

\usepackage{todonotes}

\usepackage{tikz}
\usetikzlibrary{matrix,  intersections, calc, decorations.pathreplacing} 

\usepackage{tikz-cd}

\newlength{\superthick}
\newlength{\cornerradius}
\tikzstyle{corner}=[rounded corners=\cornerradius]
\tikzstyle{dot}=[circle, inner sep=0pt, minimum size=4.8pt]
\tikzstyle{string}=[line width=\superthick]
\tikzstyle{std}=[string,dash pattern=on 0.9pt off 0.9pt]
\definecolor{realcyan}{rgb}{0,1,1}

 \usepackage{amsmath,amsthm,amsfonts,amssymb,mathrsfs,pb-diagram}
\usepackage[
bookmarks=true,colorlinks=true,linktoc=page,citecolor=darkgreen,linkcolor=darkgreen,urlcolor=darkgreen]{hyperref}
\usepackage{caption}
\usepackage{lipsum,wasysym}
\usepackage{mathtools}
\usepackage[a4paper,margin=1.1in]{geometry}
\usepackage{cleveref}
 \usepackage{amsmath}
\mathchardef\mhyphen="2D
\usepackage{color}
\usepackage{xcolor}
\usepackage{ifthen}
\usepackage{sidecap}   
\definecolor{mediumblue}{rgb}{0.0, 0.0, 0.8}

\newcommand{\Res}{{\rm Res}}

\renewcommand{\geq}{\geqslant}
\renewcommand{\leq}{\leqslant}

\tikzset{wei/.style= 
{red,double=red,double
distance=0.5pt}}

\tikzset{wei2/.style={red,double=red,double
distance=0.5pt}}

\allowdisplaybreaks
\numberwithin{equation}{section}
\usepackage{scalefnt}

\newtheorem{thm}{Theorem}[section]
\newtheorem{cor}[thm]{Corollary}
 
\newtheorem{assump}[thm]{Assumption}

\newtheorem{lem}[thm]{Lemma}
\newtheorem{prop}[thm]{Proposition}

\newtheorem*{prop*}{Proposition}
\newtheorem*{thmA*}{Theorem A}

\newtheorem*{thmB*}{Theorem B}
\newtheorem*{thmC*}{Theorem C}\newtheorem*{thm*}{Theorem D}
\newtheorem*{cor*}{Corollary}

\newtheorem*{conj*}{Conjecture A}

\newtheorem*{conj1*}{Conjecture B}
\newtheorem*{Acknowledgements*}{Acknowledgements}

\theoremstyle{rmk}

\theoremstyle{defn}
\newtheorem{rmk}[thm]{Remark}
\newtheorem{defn}[thm]{Definition}
\newtheorem{eg}[thm]{Example}

\newcommand{\rad}{\mathrm{rad}}
\newcommand{\res}{\mathrm{res}}

\newcommand{\Std}{{\rm Std}}

\newcommand{\Shape}{\operatorname{Shape}}

\newcommand{\sts}{\mathsf{s}}  
\newcommand{\stt}{\mathsf{t}}  
\newcommand{\stu}{\mathsf{u}}  
\newcommand{\stv}{\mathsf{v}}  
\newcommand{\ZZ}{{\Bbbk}}
\newcommand{\NN}{{\mathbb N}}

\newcommand{\CC}{{\mathbb{C}}}

\DeclareMathOperator{\Hom}{Hom}

\let\<=\langle
\let\>=\rangle

\newcommand\mydots{\makebox[1em][c]{\color{cyan}.\hfil\color{magenta}.\hfil\color{cyan}.\hfil\color{magenta}.}}
\tikzset{
ultra thin/.style= {line width=0.05pt},
very thin/.style=  {line width=0.2pt},
thin/.style=       {line width=0.1pt},
semithick/.style=  {line width=0.6pt},
thick/.style=      {line width=0.8pt},
very thick/.style= {line width=1.2pt},
ultra thick/.style={line width=1.6pt}
}

\crefname{defn}{Definition}{Definitions}
\crefname{thm}{Theorem}{Theorems}
\crefname{prop}{Proposition}{Propositions}
\crefname{lem}{Lemma}{Lemmas}
\crefname{cor}{Corollary}{Corollaries}
\crefname{conj}{Conjecture}{Conjectures}
\crefname{section}{Section}{Sections}
\crefname{subsection}{Subsection}{Subsections}
\crefname{eg}{Example}{Examples}
\crefname{figure}{Figure}{Figures}
\crefname{rem}{Remark}{Remarks}
\crefname{rmk}{Remark}{Remarks}
\crefname{equation}{equation}{equation}

\Crefname{defn}{Definition}{Definitions}
\Crefname{thm}{Theorem}{Theorems}
\Crefname{prop}{Proposition}{Propositions}
\Crefname{lem}{Lemma}{Lemmas}
\Crefname{cor}{Corollary}{Corollaries}
\Crefname{conj}{Conjecture}{Conjectures}
\Crefname{section}{Section}{Sections}
\Crefname{subsection}{Subsection}{Subsections}
\Crefname{eg}{Example}{Examples}
\Crefname{figure}{Figure}{Figures}
\Crefname{rem}{Remark}{Remarks}
\Crefname{rmk}{Remark}{Remarks}

\usepackage[hang,flushmargin]{footmisc}

\begin{document}

 \title  
 {The modular Weyl--Kac character formula }
 
  \author{Chris Bowman}
       \address{Department of Mathematics, 
University of York, Heslington, York, YO10 5DD, UK}
\email{Chris.Bowman-Scargill@york.ac.uk}

		\author{Amit  Hazi}
		\address{Department of Mathematics, City, University of London,   London, UK}
\email{Amit.Hazi@city.ac.uk}

 \author{Emily Norton}
 \address{University of Kent, Canterbury CT2 7NF, UK.}
\email{E.Norton@kent.ac.uk}

    \begin{abstract}
      
We classify and  explicitly construct the irreducible graded representations of anti-spherical   Hecke categories which are concentrated in one degree. Each of these homogeneous representations  
 is  one-dimensional and can be cohomologically  constructed via a   BGG resolution  involving every (infinite dimensional) standard representation  of the category.  We hence  determine   the complete first   row  of the inverse  parabolic   $p$-Kazhdan--Lusztig matrix for an arbitrary Coxeter group  and an arbitrary parabolic subgroup.  
This generalises the Weyl--Kac character formula to all Coxeter systems (and their parabolics) and proves that this generalised formula is rigid with respect to base change to   fields of arbitrary characteristic.  
   \end{abstract}
  \maketitle
  
  \section*{Introduction}

  The discovery of counterexamples to the expected bounds of Lusztig's conjecture was an earthquake in representation theory.  It marked the beginning of a new era of Lie theory, in which    diagrammatic Hecke categories  play centre stage in our attempts to understand  the structure of   algebraic groups   in terms of parabolic  ``$p$-Kazhdan--Lusztig polynomials".   
There are precious few  general results concerning   either the simple  representations of these diagrammatic Hecke categories, or the underlying  combinatorics of  parabolic $p$-Kazhdan--Lusztig polynomials.

We let $\Bbbk$ denote an algebraically closed field of   characteristic $p>2$.  
Given $ W$ an arbitrary  Coxeter group and $P$ an arbitrary parabolic subgroup, we   classify and construct   the homogeneous (simple) representations of the  anti-spherical Hecke category $\mathcal{H}_{P\backslash W}^\Bbbk$ 
(note that $\mathcal{H}_{P\backslash W}^\Bbbk$ is a $\Bbbk$-linear graded category and so this notion makes sense)
constructed from the ``geometric realisation" of $W$.  
 We prove that a  $\mathcal{H}_{P\backslash W}^\Bbbk$-module is homogeneous if and only if it is one-dimensional if and only if it is the simple $L(1_{P\backslash W})$ labelled by the identity coset $1_{P\backslash W}\in {^PW}$ (and we provide a basis of $L(1_{P\backslash W})$)   by way of Libedinsky's light leaves construction).

\begin{thmA*}
The anti-spherical Hecke category 
    $ \mathcal{H}_{P\backslash W}^\Bbbk $ admits a unique 
   homogeneous simple
     module, $L( 1_{P\backslash W})$, labelled by the identity   
     coset $1_{P\backslash W}$.  
     This simple module is a one-dimensional quotient of the infinite-dimensional standard  $ \mathcal{H}_{P\backslash W}^\Bbbk$-module $\Delta( 1_{P\backslash W})$. 
\end{thmA*}

Concurrently, we provide a  cohomological  construction of the unique homogeneous $\mathcal{H}_{P\backslash W}^\Bbbk$-module   by way of a  BGG resolution.  Within this BGG resolution, every one of the (infinite-dimensional)  standard $\mathcal{H}_{P\backslash W}^\Bbbk$-representations $\Delta( w)$ for $w\in {^PW}$     appears in  degree as dictated by the length function on the underlying Hecke algebra.   
Our BGG resolutions allow us to calculate   
 the complete first     row of the     inverse  parabolic $p$-Kazhdan--Lusztig matrix for $W$ an arbitrary Coxeter group and $P$ an arbitrary parabolic subgroup.  
 This  provides  the first   
   family of  explicit 
  (inverse)     $p$-Kazhdan--Lusztig polynomials to admit a uniform description across   arbitrary Coxeter groups  and their parabolic subgroups.   
  In the case that $W$ is an affine Weyl group and $P$ is the maximal finite parabolic subgroup 
   this gives new character formulas for representations of the corresponding algebraic groups through \cite{AJS,MR3868004}.  
 \begin{thmB*}
Associated to the unique  homogeneous simple   $ \mathcal{H}_{P\backslash W}^\Bbbk $-module, $L( 1_{P\backslash W})$,   we have a 
  complex 
  $   C_\bullet (1_{P\backslash W})=   \bigoplus_{
\begin{subarray}c 
w \in {^PW} 
\end{subarray}
}\Delta(w)\langle \ell(w)\rangle
  $
with differential given by an alternating  sum over all ``simple reflection homomorphisms".  
This complex 
 is exact except in degree zero, where    $H_0(C_\bullet( 1_{P\backslash W}))=L( 1_{P\backslash W}).$    
We hence   conclusively  generalise the Weyl--Kac character formula   to all (parabolic) Coxeter systems  
  via the formula
$$
[L( 1_{P\backslash W})] = \sum_{w\in {^PW}} (-\grade)^{\ell(w)}[\Delta(w)]
$$
and thus conclude that the first  row of the inverse parabolic $p$-Kazhdan--Lusztig matrix   has entries $(- \grade)^{\ell(w)}$ regardless of the characteristic $p\neq 2$.   
\end{thmB*}

Specialising to  the case of (affine) Weyl groups, our character formulas and resolutions have a long history.    
For finite Weyl groups,  Bernstein--Gelfand--Gelfand  constructed their eponymous resolutions in the context of finite dimensional Lie algebras \cite{bgg}.  
For Kac--Moody Lie algebras these were the subject of   Kac--Kazhdan's conjecture \cite{MR547842} (over $\mathbb{C}$) 
which was verified by Wakimoto   (for $W=\widehat{\mathfrak{S}}_2$ \cite{MR841673}),
 Hayashi (for classical type \cite{MR958588})
and Feigin,   Frenkel, and Ku  (in full generality  
  \cite{MR1011602,MR1187549})
and was extended to arbitrary fields by Mathieu \cite{MR1401425} and      subsequently   reproven by Arakawa using $W$-algebras \cite{MR2219212,MR2318558}.  
For parabolic subgroups of finite Weyl groups, our resolutions 
were first constructed in \cite{MR476813} and went on to have spectacular applications in the study of the Laplacian on Euclidean space \cite{MR2180410}.  
   For the infinite dihedral Weyl group with two generators,  these resolutions 
   were generalised to the Virasoro and blob algebras of algebraic statistical mechanics  \cite{MR995219,MS94,GJSV13}.  
For $W$ the finite symmetric group and $P$ a  maximal parabolic, these resolutions were one of the highlights of  Brundan--Stroppel's   founding work on categorical representation theory \cite{MR2600694,MR2781018}.
 For $W$ the affine symmetric group and $P$ the maximal finite parabolic  and   $\Bbbk=\CC$, Theorem~B  proves  a recent conjecture of   Berkesch--Griffeth--Sam  \cite{conjecture}
     concerning BGG resolutions of unitary modules for  Cherednik algebras.

 Kazhdan--Lusztig conjectured that much of combinatorial Lie theory should generalise beyond the realm  of Weyl groups 
 (where our resolutions admit the  geometric realisations discussed  above)   to arbitrary Coxeter groups.  
Hecke categories provide  the structural perspective in which 
the  Kazhdan--Lusztig  conjecture was finally proven  \cite{MR3245013} and serve as the archetypal  setting for studying all Lie theoretic objects.  
  In this light, Theorem B   provides the prototype for all the aforementioned 
  BGG resolutions 
  and  vastly generalises their construction to all  parabolic Coxeter systems (which are poorly understood in general, but include for example  the finite, affine, compact, paracompact, hypercompact, and Lorentzian Coxeter groups) 
      and  to    fields of positive characteristic    --- the instances  for   Lie groups, Kac--Moody  Lie algebras, and their parabolic analogues are merely the 
 examples for  which a classical geometric  structure exists.  
\color{black}

Our proof is surprisingly elementary,  
  the key idea is to exploit the 
  the  inductive nature of the
  light leaves construction.  This plays off  both the monoidal structure and the highest-weight structure of the Hecke category, which are compatible with certain truncated restriction functors.   
  Our proof is  more compact than many of the proofs for (affine) Weyl groups cited  above and should be accessible to those who are new to the research area.  
 
\color{black}

\subsection*{$p$-Kazhdan--Lusztig polynomials}
The anti-spherical  Kazhdan--Lusztig polynomials for crystallographic Coxeter systems were first studied in the language of Hecke categories by Libedinsky--Williamson over the complex numbers  \cite{antiLW}.  
The authors remark that their localisation methods do not carry over to fields of positive characteristic and so they do not define $p$-Kazhdan--Lusztig polynomials (or light leaves bases) in their paper.  
 In this paper we make the (rather trivial) observation that one does not need such localisation methods in order to define  $p$-Kazhdan--Lusztig polynomials (or light leaves bases).  
 In \cref{lighteaves} we discuss under what circumstances one can restrict to $\mathbb Z \subset \CC$ and hence obtain light leaves bases by ``reduction modulo $p$". In this manner, we define the $p$-Kazhdan--Lusztig polynomials for arbitrary (parabolic) Coxeter systems.  
 
With this machinery in place, we note that Theorem B provides the first instance of a complete row/column of the (inverse) ($p$-)Kazhdan--Lusztig matrix to be calculated for  all Coxeter groups (and their parabolics).  
This is especially noteworthy considering just how difficult it is to calculate $p$-Kazhdan--Lusztig polynomials --- for example, the current state-of-the-art ``billiards conjecture'' of Jensen--Lusztig--Williamson describes an infinitesimally small region of  the $p$-Kazhdan--Lusztig matrix   in type ${A}_2 \backslash \widehat{A}_2$ \cite{MR3766576,2105.04665}.
 The only other examples of $p$-Kazhdan--Lusztig polynomials which are explicitly known are the famous ``torsion explosion" counter examples to the Lusztig conjecture \cite{w13}.  
 Therefore we expect our examples of (characteristic-free) $p$-Kazhdan--Lusztig polynomials to be of wide interest in their community,  
 as they provide the most general family of these polynomials calculated to date.

\section{The diagrammatic Hecke categories}    

We begin by recalling the basics of   diagrammatic Hecke categories.  
Almost everything from this section is  lifted from   Elias--Williamson's original paper \cite{MR3555156}  or is an extension of their results to the parabolic setting \cite{antiLW}.  

  \begin{rmk} 
The cyclotomic quotients of (anti-spherical) Hecke categories are small categories  
with finite-dimensional morphism spaces given by the light leaves basis of \cite{MR3555156,antiLW}.     
Working with such a category is equivalent working to with a locally unital algebra, as defined in \cite[Section 2.2]{BSBS}, see  \cite[Remark 2.3]{BSBS}.  
   Throughout this paper we will work in the latter setting.  
   The reader who prefers to think of categories can equivalently phrase everything in the this paper  in terms of categories and representations of categories.

   \end{rmk}

 \color{black} 
\subsection{Coxeter systems} 
Let $(W, S)$ be a Coxeter system:  $W$ is the group generated by the finite set $S$ subject to the relations $(\csigma\ctau)^{m_{\csigma\ctau}} = 1$ for  
 $ {m_{\csigma\ctau}}\in \NN\cup\{\infty\}$ 
 satisfying ${m_{\csigma\ctau}}= {m_{\ctau\csigma}}$ and that  ${m_{\csigma\ctau}}=1$ if and only if  $\csigma= \ctau\in S$.  
   Let $\ell : W \to \mathbb{N}$  be the corresponding length function. 
   Let $\mathcal{L} = \mathbb{Z}[\grade,\grade^{-1}]$ be the ring of Laurent polynomials with integer coefficients in one variable $\grade$.

Consider $S_P \subseteq S$ an arbitrary subset and $(P, S_P)$
 its corresponding Coxeter group. We say that $(P, S_P)$ is the parabolic subgroup corresponding to the set $S_P\subseteq S$.
We denote by $^ PW \subseteq W$ the set of minimal coset representatives in $P\backslash W$. 

Given an expression $\w=\sigma_1\sigma_2\cdots  \sigma_\ell$ for $\sigma_i\in W$ for $1\leq i \leq \ell$, we let $w$ be the corresponding element of $W$.  
     We  define a subword of  $\w$ to be a sequence
${\sf t}=(t_1,t_2,\dots ,t_\ell)\in\{0,1\}^\ell$ and we set 
 $\w^\stt  :=\sigma_1^{t_1}\sigma_2^{t_2}\cdots \sigma_\ell^{t_\ell}$ and we emphasise that $s_i^0=1_W  \in W$.   
 We write  $y\leq w$ if for some (or equivalently, every) 
reduced expression $\w$ there exists   a subword $\stt$ and a reduced expression  $\y$ such that $\w^\stt=\y$.  
  We let ${\sf exp}_P^\ell(w)$ denote the set of all expressions $\w=\sigma_1\sigma_2\cdots  \sigma_\ell$ of $w$ of length $\ell$ such that   $\sigma_1\dots\sigma_k \in {^PW}$ for each $1\leq k \leq \ell$, we let ${\sf exp}_P (w)=\cup_{\ell\geq 0}{\sf exp}_P^\ell(w)$, and 
$ {\sf exp}_P^{\ell}=\cup_{w \in W}{\sf exp}_P^\ell(w)$.    
 We set 
  ${\sf rexp}_P (w):= {\sf exp}_P^{\ell(w)} (w)$.

\subsection{Bi-coloured quantum numbers and Cartan matrices  }

We define the $x$- and $y$- bicoloured quantum numbers as follows.  First set 
\begin{equation}
[0]_x=[0]_y=0 \quad [1]_x=[1]_y=1 \quad [2]_x=x \quad [2]_y=y
\end{equation}
and then inductively define 
\begin{equation}
[2]_x[k]_y=[k+1]_x+[k-1]_x\qquad 
[2]_y[k]_x=[k+1]_y+[k-1]_y .
\end{equation}
 When $k$ is odd, $[k]_x=[k]_y$.  
The following  definition allows one to speak of Cartan matrices of  Coxeter groups.  

\begin{defn}
Let $\Bbbk$ be a complete local ring in which $2$ is invertible. A {\sf balanced Cartan matrix} of $(W,S)$ over $\Bbbk$ is an $|S| \times |S|$-matrix $(a_{\csigma \ctau})_{\csigma,\ctau \in S}$ such that
\begin{enumerate}
\item for all $\csigma \in S$ we have $a_{\csigma\csigma}=2$;
\item for all distinct $\csigma,\ctau  \in S$ such that $m_{\csigma \ctau}<\infty$, set $x=-a_{\csigma \ctau}$ and $y=-a_{\ctau \csigma}$.  
We require that 
\begin{equation} [m_{\csigma\ctau}]_x = [m_{\csigma\ctau}]_y = 0\qquad [m_{\csigma\ctau}-1]_x = [m_{\csigma\ctau}-1]_y = 1
\end{equation}
\end{enumerate}
\end{defn}
 
%

\subsection{Soergel graphs }
Let $(W,S)$ denote an arbitrary Coxeter system with $S$ finite.   
Given $\csigma\in S\cup \{\emptyset\}$ we define the {\sf monochrome Soergel generators} to be the framed graphs 
$$
{\sf 1}_{\emptyset }  =
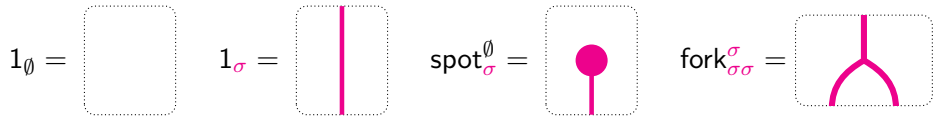
\begin{minipage}{1.5cm} \begin{tikzpicture}[scale=1.2]
\draw[densely dotted,rounded corners](-0.5cm,-0.6cm)  rectangle (0.5cm,0.6cm);
\clip(0,0) circle (0.6cm);
\end{tikzpicture}
\end{minipage}
\quad
{\sf 1}_{\csigma }  = 
\begin{minipage}{1.5cm} \begin{tikzpicture}[scale=1.2]
\draw[densely dotted,rounded corners](-0.5cm,-0.6cm)  rectangle (0.5cm,0.6cm);
\clip(0,0) circle (0.6cm);
\draw[line width=0.06cm, magenta](0,-1)--(0,+1);
\end{tikzpicture}
\end{minipage}
\quad
   {\sf spot}_\csigma^\emptyset
  =
  \begin{minipage}{1.5cm}  
  \begin{tikzpicture}[scale=1.2]
\draw[densely dotted,rounded corners](-0.5cm,-0.6cm)  rectangle (0.5cm,0.6cm);

\clip(0,0) circle (0.6cm);
\draw[line width=0.06cm, magenta](0,-1)--(0,+0);
\fill[magenta] (0,0)  circle (5pt);
\end{tikzpicture}\end{minipage}
\quad 
{\sf fork}_{\csigma\csigma}^{\csigma}= 
   \begin{minipage}{2cm}  
%
\begin{tikzpicture}[scale=1.2]
\draw[densely dotted,rounded corners](-0.75cm,-0.5cm)
  rectangle (0.75cm,0.5cm);
  \clip (-0.75cm,-0.5cm)
  rectangle (0.75cm,0.5cm);
 \draw[line width=0.08cm, magenta](0,0)to [out=-30, in =90] (10pt,-15pt);
 \draw[line width=0.08cm, magenta](0,0)to [out=-150, in =90] (-10pt,-15pt);
 \draw[line width=0.08cm, magenta](0,0)--++(90:1);
 \end{tikzpicture}
\end{minipage}$$ and given any $\csigma,\ctau\in S $ with $m_{\sigma\tau}=m<\infty$ we have the {\sf bi-chrome generator}
$$
{\sf braid}_{\csigma\ctau }^{ \ctau\csigma }(m)
=
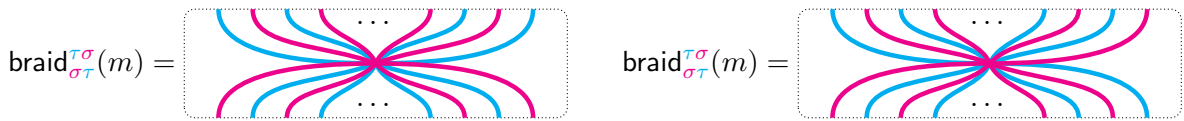
\begin{minipage}{5cm}
\begin{tikzpicture}[yscale=1.2,xscale=1.8]
\draw[densely dotted,rounded corners](-1.4cm,-0.6cm)
  rectangle (1.4cm,0.6cm);

   \draw[line width=0.06cm, cyan](0,0)to [out=180, in =-90] (-0.15+-1 ,0.6);
    \draw[line width=0.06cm , cyan](0,0)to [out=120, in =-90] (-0.15+-0.5 ,0.6);
   \draw[line width=0.06cm , magenta](0,0)to [out=105, in =-90] (-0.15+-0.25 ,0.6);   
    \draw[line width=0.06cm , magenta](0,0)to [out=150, in =-90] (-0.15+-0.75 ,0.6);

     \draw[line width=0.06cm, cyan](0,0)to [out=0, in =-90] (0.15+1 ,0.6);
    \draw[line width=0.06cm , cyan](0,0)to [out=60, in =-90] ( 0.15+ 0.5 ,0.6);
   \draw[line width=0.06cm , magenta](0,0)to [out=75, in =-90] ( 0.15+ 0.25 ,0.6);   
    \draw[line width=0.06cm , magenta](0,0)to [out=30, in =-90] ( 0.15+ 0.75 ,0.6);

     \draw[line width=0.06cm , cyan](0,0)to [out=-105, in=90] (-0.15+-0.25 ,-0.6);   
    \draw[line width=0.06cm , cyan](0,0)to [out=-150, in=90] (-0.15+-0.75 ,-0.6);

    \draw[line width=0.06cm , cyan](0,0)to [out=-75, in=90] ( 0.15+ 0.25 ,-0.6);   
    \draw[line width=0.06cm , cyan](0,0)to [out=-30, in=90] ( 0.15+ 0.75 ,-0.6);
     \draw[line width=0.06cm, magenta](0,0)to [out=-0, in=90] (0.15+1 ,-0.6);
    \draw[line width=0.06cm , magenta](0,0)to [out=-60, in=90] ( 0.15+ 0.5 ,-0.6);
    \draw[line width=0.06cm, magenta](0,0)to [out=-180, in=90] (-0.15+-1 ,-0.6);
    \draw[line width=0.06cm , magenta](0,0)to [out=-120, in=90] (-0.15+-0.5 ,-0.6);

   \draw(0,0.45) node {$\dots$}; 
      \draw(0,-0.45) node {$\dots$}; 
 \end{tikzpicture} \end{minipage}
 \qquad
 {\sf braid}_{\csigma\ctau }^{ \ctau\csigma }(m)
=
\begin{minipage}{5cm}
\begin{tikzpicture}[yscale=1.2,xscale=1.8]
\draw[densely dotted,rounded corners](-1.4cm,-0.6cm)
  rectangle (1.4cm,0.6cm);

        \draw[line width=0.06cm , cyan](0,0)to [out=-105, in=90] (-0.15+-0.25 ,-0.6);   
    \draw[line width=0.06cm , cyan](0,0)to [out=-150, in=90] (-0.15+-0.75 ,-0.6);

     \draw[line width=0.06cm, cyan](0,0)to [out=-0, in=90] (0.15+1 ,-0.6);
    \draw[line width=0.06cm , cyan](0,0)to [out=-60, in=90] ( 0.15+ 0.5 ,-0.6);

   \draw[line width=0.06cm, cyan](0,0)to [out=180, in =-90] (-0.15+-1 ,0.6);
    \draw[line width=0.06cm , cyan](0,0)to [out=120, in =-90] (-0.15+-0.5 ,0.6);

   \draw[line width=0.06cm , cyan](0,0)to [out=75, in =-90] ( 0.15+ 0.25 ,0.6);   
    \draw[line width=0.06cm , cyan](0,0)to [out=30, in =-90] ( 0.15+ 0.75 ,0.6);

     \draw[line width=0.06cm, magenta](0,0)to [out=0, in =-90] (0.15+1 ,0.6);
    \draw[line width=0.06cm , magenta](0,0)to [out=60, in =-90] ( 0.15+ 0.5 ,0.6);

\draw[line width=0.06cm , magenta](0,0)to [out=105, in =-90] (-0.15+-0.25 ,0.6);   
    \draw[line width=0.06cm , magenta](0,0)to [out=150, in =-90] (-0.15+-0.75 ,0.6);

     \draw[line width=0.06cm, magenta](0,0)to [out=-180, in=90] (-0.15+-1 ,-0.6);
    \draw[line width=0.06cm , magenta](0,0)to [out=-120, in=90] (-0.15+-0.5 ,-0.6);

    \draw[line width=0.06cm , magenta](0,0)to [out=-75, in=90] ( 0.15+ 0.25 ,-0.6);   
    \draw[line width=0.06cm , magenta](0,0)to [out=-30, in=90] ( 0.15+ 0.75 ,-0.6);
  
   \draw(0,0.45) node {$\dots$}; 
      \draw(0,-0.45) node {$\dots$}; 
 \end{tikzpicture} \end{minipage}
$$
for $m$ odd, or even respectively.  Here
 the northern edges     are  coloured with the sequence 
\begin{align*}  \underbrace{\ctau\csigma\ctau\csigma\dots \csigma\ctau }_{m \text{ times}} 
\qquad\underbrace{\ctau\csigma\ctau\csigma\dots \ctau\csigma}_{m \text{ times}} 
\end{align*}
for  $m$ odd or even respectively and the southern edges are coloured 
\begin{align*}
  \underbrace{\csigma\ctau\csigma\ctau\dots \ctau\csigma }_{m \text{ times}}
\quad
\underbrace{\csigma\ctau\csigma\ctau\dots \csigma\ctau}_{m \text{ times}}
\end{align*}
for  $m$ odd or even respectively.  We define the northern/southern reading word of a Soergel generator (or its dual) to be word in the alphabet $S$ obtained by reading the colours of the northern/southern edge of the frame respectively (ignoring any $\emptyset$ symbols).  
Pictorially, we define the duals of these generators to be the graphs obtained by reflection through their horizontal axes.  Non-pictorially, we simply swap the sub- and superscripts.  We   denote this duality by $\ast$.  For example, the dual of the fork generator is pictured as follows
\begin{equation}
  {\sf fork}^{\csigma\csigma}_{\csigma}=  
  \begin{minipage}{2cm}  
 \begin{tikzpicture}[scale=-1.2]
 \draw[densely dotted,rounded corners](-0.75cm,-0.5cm)
  rectangle (0.75cm,0.5cm);
  \clip (-0.75cm,-0.5cm)
  rectangle (0.75cm,0.5cm);
  \draw[line width=0.08cm, magenta](0,0)to [out=-30, in =90] (10pt,-15pt);
 \draw[line width=0.08cm, magenta](0,0)to [out=-150, in =90] (-10pt,-15pt);
 \draw[line width=0.08cm, magenta](0,0)--++(90:1);
 \end{tikzpicture}
\end{minipage}= ( {\sf fork}_{\csigma\csigma}^{\csigma})^\ast.
\end{equation}
 Given any two (dual) Soergel generators $D$ and $D'$ we define $D\otimes D'$ to be the diagram obtained by horizontal concatenation (and we extend this linearly).  The northern/southern colour sequence of $D\otimes D'$ is  the concatenation of those of $D$ and $D'$ ordered from left to right.    
Given any two (dual) Soergel generators, we define their product $D\circ D'$ to be the vertical concatenation of $D$ on top of $D'$ if the southern reading word of $D$ is equal to 
 the northern reading word of $D'$ and to be zero otherwise.  
Finally, we define a {\sf Soergel graph} to be any graph obtained by repeated  horizontal and vertical concatenation of the Soergel generators and their duals.

      \subsection{Some specific graphs} 
For $\w=\sigma_1\dots\sigma_\ell$ an expression, we define 
\begin{align}\label{itemmmmmmmmm}
 {\sf 1}_\w&={\sf1}_{\sigma_1}\otimes {\sf1}_{\sigma_2}\otimes \dots\otimes {\sf1}_{\sigma_\ell} 
\intertext{and given $k>1$ and $\csigma,\ctau\in S$ we set }
  {\sf 1}_{ \csigma\ctau} ^{k } & = 
{\sf 1}_{ \csigma}\otimes {\sf 1}_\ctau 
\otimes{\sf 1}_{ \csigma}\otimes {\sf 1}_\ctau \dots 
\end{align}
to be the alternately coloured idempotent on $k$ strands (so that the final strand is  $\csigma$- or $\ctau$-coloured if $k$ is odd or even respectively).     Given $\csigma,\ctau \in S$ with $m_{\csigma\ctau}=m$ even, let  
\begin{equation}\w=\rho_1\cdots \rho_k (\csigma\ctau\cdots\csigma \ctau) \rho_{m+k+1}\cdots \rho_\ell\qquad 
\underline{\w}=\rho_1\cdots \rho_k (\ctau\csigma\cdots  \ctau\csigma) \rho_{m+k+1}\cdots \rho_\ell
\end{equation}  
be two reduced expressions for $w\in W$.   We say that $\w$ and $\underline{\w}$ are {\sf adjacent} and we set 
\begin{equation}{\sf braid}^\w_{\underline{\w}}={\sf 1}_{\rho_1}\otimes \cdots \otimes {\sf 1}_{ \rho_k}
\otimes 
{\sf braid}^{\csigma\ctau}_{\ctau\csigma}(m) 
\otimes {\sf 1}_{\rho_{m+k+1}}\otimes  \cdots \otimes {\sf 1}_{ \rho_\ell}
\end{equation}
(similarly for $m$ odd).  
Now, given a sequence of adjacent reduced expressions, 
\begin{equation}\w=\w^{(1)}, \w^{(2)},\dots, \w^{(q)}=\underline{\w}
\end{equation} 
and the value $q$ is minimal such that 
this sequence exists, then we set 
\begin{equation}{\sf braid}^\w_{\underline{\w}}
=
\prod _{1\leq p < q}{\sf braid}^{\w^{(p)}}_{{\w^{(p+1)}} }
\in \mathcal{H}_W^\Bbbk .
\end{equation}
While this element is not uniquely defined,  
 only the minimality will  matter for our purposes.

\begin{eg}
\color{black}The left and righthand  diagrams depicted in \cref{Katya} 
are both of the form  ${\sf braid}^\w_{\underline{\w}}$
 for \begin{align}\w={\gam\al\gam{\bet\al\bet} {\gam\al\bet}} \quad \text{ and } \quad 
  \underline{\w} = {\bet\al\gam\bet\al\bet} .
  \end{align}  The   corresponding sequences 
 of adjacent reduced expressions are recorded in 
 the Zamolodchikov relation \cref{Zamolodchikov relations}.  
\end{eg} 
 
    \subsection{The diagrammatic Hecke categories }
   Let $(W,S)$ be a Coxeter system with a balanced Cartan matrix 
  $(a_{\csigma \ctau}) _{\csigma,\ctau \in S}$.  
  Suppose $\csigma,\ctau \in S$ with $m=m_{\csigma \ctau}<\infty$. 
In order to save space, we set 
\begin{equation}
 {\sf  jw} ^{\csigma\ctau\csigma}_{\csigma\csigma\csigma}= 
  ( {\sf 1}_\csigma\otimes {\sf spot}^\ctau_\emptyset \otimes  {\sf 1}_\csigma)  
({\sf fork}_{\csigma\csigma}^{\csigma}\otimes {\sf 1}_{\csigma}	)
\qquad
 {\sf   wj} _{\csigma\ctau\csigma}^{\csigma\csigma\csigma}= 
( {\sf 1}_{\csigma} \otimes {\sf fork}^{\csigma\csigma}_{\csigma} 	)
  ( {\sf 1}_\csigma \otimes {\sf spot}_\ctau^\emptyset \otimes  {\sf 1}_\csigma)  
\end{equation}
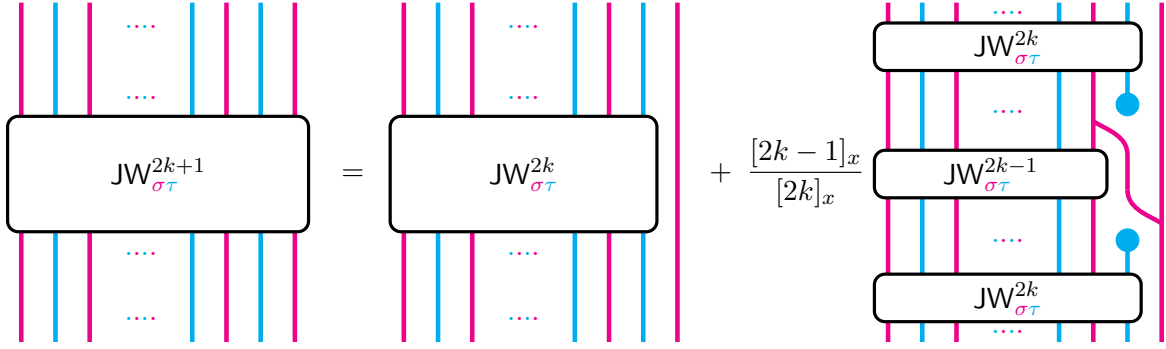
  We are now ready to inductively  define the  Jones--Wenzl projector  
  ${\sf JW}_{\csigma \ctau}^{2k+1}$ 
to be the element  
$$      {\sf JW}_{\csigma \ctau}^{2k}
 \otimes {\sf 1}_\csigma 
 +
 \tfrac{[2k-1]_x}{[2k]_x}
 ( {\sf JW}_{\csigma \ctau}^{2k}\otimes  {\sf 1}_\csigma )
(  {\sf 1}_{ \csigma\ctau}^{{2k-2} }\otimes {\sf jw} ^{\csigma\ctau\csigma}_{\csigma\csigma\csigma})
 ( {\sf JW}_{\csigma \ctau}^{2k-1}\otimes  {\sf 1}_{\csigma\csigma} )
(  {\sf 1}_{ \csigma\ctau }^{{2k-2} }\otimes {\sf wj} _{\csigma\ctau\csigma}^{\csigma\csigma\csigma}) ( {\sf JW}_{\csigma \ctau}^{2k}\otimes  {\sf 1}_\csigma )
$$
and the   Jones--Wenzl projector  
  ${\sf JW}_{\csigma \ctau}^{2k}$ to be the element  
$$
      {\sf JW}_{\csigma \ctau}^{2k-1}
 \otimes {\sf 1}_\ctau 
 +
 \tfrac{[2k-2]_x}{[2k-1]_x}
 ( {\sf JW}_{\csigma \ctau}^{2k-1}\otimes  {\sf 1}_\ctau )
({\sf 1}_{ \csigma\ctau}^{{2k-3}} \otimes {\sf jw} ^{ \ctau\csigma\ctau}_{\ctau\ctau\ctau})
 ( {\sf JW}_{\csigma \ctau}^{2k-2}\otimes  {\sf 1}_{\ctau\ctau} )
(  {\sf 1}_{ \csigma\ctau}^{{2k-3}}
 \otimes   {\sf wj} _{ \ctau\csigma\ctau}^{\ctau\ctau\ctau})
( {\sf JW}_{\csigma \ctau}^{2k-1}\otimes  {\sf 1}_\ctau ) . $$
  We remark that in each case the leftmost strand is coloured with 
  $\csigma$ and the second term  has coefficient equal to a ratio of $x$-bicoloured  quantum integers. 
  The pictorial version of the first recursion (for $2k+1$ odd) is 
  as follows
 $$\begin{minipage}{4.2cm}  
 \begin{tikzpicture}[scale=0.9]

\draw[cyan,line width=0.06cm,fill=cyan](1,-1)--++(90:1.85) circle (4pt);
\draw[cyan,line width=0.06cm,fill=cyan](1,4)--++(-90:1.85) circle (4pt);
\draw[magenta,line width=0.06cm](-2.5,-1)--++(90:5);
\draw[cyan,line width=0.06cm](-2,-1)--++(90:5);
\draw[cyan,line width=0.06cm](0,-1)--++(90:5);
\draw[magenta,line width=0.06cm](-1.5,-1)--++(90:5);

\draw(-0.75,3.85-0.2) node {$\mydots$};
\draw(-0.75,3.85-1.45+0.2) node {$\mydots$};
\draw(-0.75,3.85-1.5-1.85-0.2) node {$\mydots$};
\draw(-0.75,-1+0.15+0.2) node {$\mydots$};

 \draw[magenta,line width=0.06cm ](0.5,2.5-0.25) to  [in=90,out=-30] (1,2-0.25)--(1,1.25);
 \draw[magenta,line width=0.06cm ](1,1+0.25) to  [out=-90,in=150] (1.5,0.5+0.25);

\draw[magenta,line width=0.06cm](1.5,-1)--++(90:5);
\draw[magenta,line width=0.06cm](0.5,-1)--++(90:5);

\draw[very thick, rounded corners,fill=white] (-2.7,.125+1.7+0.5) rectangle (1.2+0.5,.125+1-0.5) node[pos=0.5] { ${\sf JW}_{\csigma\ctau}^{2k+1}$ };

 \end{tikzpicture}
 \end{minipage}
\; =\;\;
\begin{minipage}{4.2cm}  
 \begin{tikzpicture}[scale=0.9]

\draw[cyan,line width=0.06cm,fill=cyan](1,-1)--++(90:1.85) circle (4pt);
\draw[cyan,line width=0.06cm,fill=cyan](1,4)--++(-90:1.85) circle (4pt);
\draw[magenta,line width=0.06cm](-2.5,-1)--++(90:5);
\draw[cyan,line width=0.06cm](-2,-1)--++(90:5);
\draw[cyan,line width=0.06cm](0,-1)--++(90:5);
\draw[magenta,line width=0.06cm](-1.5,-1)--++(90:5);

\draw(-0.75,3.85-0.2) node {$\mydots$};
\draw(-0.75,3.85-1.45+0.2) node {$\mydots$};
\draw(-0.75,3.85-1.5-1.85-0.2) node {$\mydots$};
\draw(-0.75,-1+0.15+0.2) node {$\mydots$};


\draw[magenta,line width=0.06cm](1.5,-1)--++(90:5);
\draw[magenta,line width=0.06cm](0.5,-1)--++(90:5);

\draw[very thick, rounded corners,fill=white] (-2.7,.125+1.7+0.5) rectangle (1.2 ,.125+1-0.5) node[pos=0.5] { ${\sf JW}_{\csigma\ctau}^{2k}$ };

 \end{tikzpicture}
 \end{minipage}\! + \;  \frac{[2k-1]_x}{[2k]_x}\;   \begin{minipage}{4.2cm}  
 \begin{tikzpicture}[scale=0.9]

\draw[cyan,line width=0.06cm,fill=cyan](1,-1)--++(90:1.5) circle (4pt);
\draw[cyan,line width=0.06cm,fill=cyan](1,4)--++(-90:1.5) circle (4pt);
\draw[magenta,line width=0.06cm](-2.5,-1)--++(90:5);
\draw[cyan,line width=0.06cm](-2,-1)--++(90:5);
\draw[cyan,line width=0.06cm](0,-1)--++(90:5);
\draw[magenta,line width=0.06cm](-1.5,-1)--++(90:5);

\draw(-0.75,3.85) node {$\mydots$};
\draw(-0.75,3.85-1.45) node {$\mydots$};
\draw(-0.75,3.85-1.5-1.85) node {$\mydots$};
\draw(-0.75,-1+0.15) node {$\mydots$};

 \draw[magenta,line width=0.06cm ](0.5,2.5-0.25) to  [in=90,out=-30] (1,2-0.25)--(1,1.25);
 \draw[magenta,line width=0.06cm ](1,1+0.25) to  [out=-90,in=150] (1.5,0.5+0.25);

\draw[magenta,line width=0.06cm](1.5,-1)--++(90:5);
\draw[magenta,line width=0.06cm](0.5,-1)--++(90:5);

\draw[very thick, rounded corners,fill=white] (-2.7,3.7) rectangle (1.2,3) node[pos=0.5] { ${\sf JW}_{\csigma\ctau}^{2k}$ };

\draw[very thick, rounded corners,fill=white] (-2.7,.125+1.7) rectangle (1.2-0.5,.125+1) node[pos=0.5] { ${\sf JW}_{\csigma\ctau}^{2k-1}$ };

\draw[very thick, rounded corners,fill=white] (-2.7,3.7-3.7) rectangle (1.2,3-3.7) node[pos=0.5] { ${\sf JW}_{\csigma\ctau}^{2k}$ };

 \end{tikzpicture}.
 \end{minipage}$$
       The elements 
  ${\sf JW}_{  \ctau\csigma}^{2k }$ 
  and 
    ${\sf JW}_{  \ctau\csigma}^{2k+1 }$ are the same as the above except with the inverted colour pattern and coefficients equal to   $y$-bicoloured  quantum integers.  
   Specifically, we set $  {\sf JW}_{\ctau \csigma}^{2k+1}$ to be the element 
     $$
      {\sf JW}_{\ctau \csigma}^{2k}
 \otimes {\sf 1}_\ctau 
 +
 \tfrac{[2k-1]_y}{[2k]_y}
 ( {\sf JW}_{\ctau \csigma}^{2k}\otimes  {\sf 1}_\ctau )
(  {\sf 1}_{ \ctau\csigma}^{{2k-2} }\otimes {\sf jw} ^{\ctau\csigma\ctau}_{\ctau\ctau\ctau})
 ( {\sf JW}_{\ctau \csigma}^{2k-1}\otimes  {\sf 1}_{\ctau\ctau} )
(  {\sf 1}_{ \ctau\csigma }^{{2k-2} }\otimes {\sf wj} _{\ctau\csigma\ctau}^{\ctau\ctau\ctau}) ( {\sf JW}_{\ctau \csigma}^{2k}\otimes  {\sf 1}_\ctau )
  $$
and we set $  {\sf JW}_{\ctau \csigma}^{2k }$ to be the element 
  $$
      {\sf JW}_{\ctau \csigma}^{2k-1}
 \otimes {\sf 1}_\csigma 
 +
 \tfrac{[2k-2]_y}{[2k-1]_y}
 ( {\sf JW}_{\ctau \csigma}^{2k-1}\otimes  {\sf 1}_\csigma )
(  {\sf 1}_{ \ctau\csigma}^{{2k-3}}
 \otimes {\sf jw} ^{ \csigma\ctau\csigma}_{\csigma\csigma\csigma})
 ( {\sf JW}_{\ctau \csigma}^{2k-2}\otimes  {\sf 1}_{\csigma\csigma}  )
(  {\sf 1}_{ \ctau\csigma}^{{2k-3}}
 \otimes   {\sf  wj} _{ \csigma\ctau\csigma}^{\csigma\csigma\csigma})
( {\sf JW}_{\ctau \csigma}^{2k-1}\otimes  {\sf 1}_\csigma ) . $$
\color{black}
Finally, we define ${\sf JW}_{\csigma\ctau}$ and ${\sf JW}_{\ctau \csigma}$ to be the evaluation of the diagrams ${\sf JW}_{\csigma \ctau}^m$ and ${\sf JW}_{\ctau \csigma}^m$ respectively at $x=-a_{\csigma\ctau}$ and $y=-a_{\csigma \ctau}$.\footnote{Here we are using the fact that Jones--Wenzl projectors can be can be computed ``generically'' \cite[Theorem~6.13]{ew-localizedcalc}.}

\color{black} 
 
 
   \begin{defn}   \renewcommand{\vvv}{{\underline{w} }} 
\renewcommand{\w}{{\underline{x}}}
\renewcommand{\x}{{\underline{y}}}
\renewcommand{\y}{{\underline{z}}}
Let $\Bbbk$ be an arbitrary commutative ring.   
Let $(W,S)$ be a Coxeter system with a balanced Cartan matrix $(a_{\csigma \ctau})_{\csigma,\ctau \in S}$ over a commutative ring $\Bbbk$.  We define  $\mathcal{H}^\Bbbk_W$ to be the locally-unital  associative $\mathbb Z$-graded $\Bbbk$-algebra spanned by all Soergel-graphs, with duality $\ast$, and  multiplication given by vertical concatenation of diagrams modulo the following local relations and their duals.   
\begin{align}
{\sf 1}_{\csigma} {\sf 1}_{\ctau}& =\delta_{\csigma,\ctau}{\sf 1}_{\csigma} & {\sf 1}_{\emptyset} {\sf 1}_{\csigma} & =0 & {\sf 1}_{\emptyset}^2& ={\sf 1}_{\emptyset}\\
{\sf 1}_{\emptyset} {\sf spot}_{\csigma}^\emptyset {\sf 1}_{\csigma}& ={\sf spot}_{\csigma}^{\emptyset} & {\sf 1}_{\csigma} {\sf fork}_{\csigma\csigma}^{\csigma} {\sf 1}_{\csigma\csigma}& ={\sf fork}_{\csigma\csigma}^{\csigma} &
{\sf 1}_{\ctau\csigma}^m {\sf braid}_{\csigma\ctau}^{\ctau\csigma}(m) {\sf 1}_{\csigma\ctau}^m & ={\sf braid}_{\csigma\ctau}^{\ctau\csigma}(m)
\end{align}
For each  $\csigma \in S $  we have  monochrome relations 
\begin{align}
({\sf spot}_\csigma^\emptyset \otimes {\sf 1}_\csigma){\sf fork}^{\csigma\csigma}_{\csigma}
=
{\sf 1}_{\csigma}
 \qquad\quad  
  ({\sf 1}_\csigma\otimes {\sf fork}_{\csigma\csigma}^{ \csigma} )
({\sf fork}^{\csigma\csigma}_{\csigma}\otimes {\sf 1}_{\csigma})
=
{\sf fork}^{\csigma\csigma}_{\csigma}
{\sf fork}^{\csigma}_{ \csigma\csigma}
\end{align}
\begin{align}{\sf fork}_{\csigma\csigma}^{\csigma}
{\sf fork}^{\csigma\csigma}_{\csigma}=0 \qquad \quad ({\sf spot}_\csigma^\emptyset {\sf spot}^\csigma_\emptyset)\otimes {\sf 1}_{\csigma}
+
 {\sf 1}_{\csigma} \otimes ({\sf spot}_\csigma^\emptyset {\sf spot}^\csigma_\emptyset)
 =
2 ({\sf spot}^\csigma_\emptyset {\sf spot}_\csigma^\emptyset)
\end{align}
For every ordered  pair  $(\csigma,\ctau) \in S^2$ with $\csigma \neq \ctau$,   the   bi-chrome relations:  The two-colour barbell, 
\begin{align}({\sf spot}_\ctau^\emptyset {\sf spot}^\ctau_\emptyset)\otimes {\sf 1}_{\csigma}
-
 {\sf 1}_{\csigma} \otimes ({\sf spot}_\ctau^\emptyset {\sf spot}^\ctau_\emptyset)
 = a_{\csigma \ctau} (({\sf spot}^\csigma_{\emptyset} {\sf spot}_\csigma^\emptyset )
   -
 {\sf 1}_{\csigma} \otimes ({\sf spot}_\csigma^\emptyset {\sf spot}^\csigma_\emptyset))
.\end{align}
 If $m=m_{\csigma\ctau}<\infty$ we also have  the fork-braid relations 
\begin{align}
 {\sf braid}_{\csigma \ctau\cdots \ctau \csigma}^{\ctau\csigma\cdots \csigma\ctau} 
({\sf fork}^{\csigma}_{\csigma\csigma}\otimes{\sf 1}_{ \ctau\csigma} ^{m-1} 	)
 ( {\sf 1}_\csigma \otimes{\sf braid}^{  \csigma \ctau\cdots \ctau \csigma}_
 {  \ctau\csigma\cdots \csigma\ctau} )
=
   ({\sf 1}_{ \ctau\csigma}^{m-1}   \otimes {\sf fork}_{\ctau\ctau}^{ \ctau})
(
 {\sf braid}_{ \csigma \ctau\cdots \ctau \csigma}^{  \ctau\csigma\cdots \csigma\ctau} 
\otimes {\sf 1}_\ctau) 
\\
 {\sf braid}_{\csigma \ctau\cdots  \csigma \ctau}^{\ctau\csigma\cdots \ctau \csigma} 
({\sf fork}^{\csigma}_{\csigma\csigma}\otimes{\sf 1}_{\ctau\csigma} ^{m-1} 	)
 ( {\sf 1}_\csigma \otimes{\sf braid}^{  \csigma \ctau\cdots  \csigma\ctau }_
 {  \ctau\csigma\cdots  \ctau\csigma} )
=
   ({\sf 1}_{ \ctau\csigma}^{m-1}   \otimes {\sf fork}_{\csigma\csigma}^{\csigma})
( 
  {\sf braid}_{  \csigma \ctau\cdots  \csigma \ctau}^{  \ctau\csigma\cdots \ctau \csigma} 
\otimes {\sf 1}_\csigma) 
\end{align} 
for $m$ odd and even, respectively.  
We   require the cyclicity  relation,  
\begin{align}
({\sf 1}_{ \ctau\csigma}^{m }\otimes ({\sf spot}_{\csigma}^\emptyset {\sf fork}^{\csigma}_{\csigma\csigma}))({\sf 1}_\ctau \otimes
 {\sf braid}^{\csigma\ctau}_{ \ctau\csigma}(m)  
\otimes {\sf 1}_\csigma) 
( (  {\sf fork}_{\ctau}^{\ctau\ctau} {\sf spot}^{\ctau}_\emptyset)\otimes 
{\sf 1}_{ \csigma \ctau}^{m })
&=
  {\sf braid}^{\ctau \csigma\dots \csigma \ctau}_{\csigma \ctau\cdots \ctau\csigma} 
\\
({\sf 1}_{ \ctau\csigma}^{m }\otimes( {\sf spot}_{\ctau}^\emptyset {\sf fork}^{\ctau}_{\ctau\ctau}))({\sf 1}_\ctau \otimes
 {\sf braid}^{\csigma\ctau}_{ \ctau\csigma}(m)  
\otimes {\sf 1}_\ctau) 
( (  {\sf fork}_{\ctau}^{\ctau\ctau} {\sf spot}^{\ctau}_\emptyset)\otimes 
{\sf 1}_{ \csigma \ctau}^{m })
&=
  {\sf braid}^{\ctau \csigma\dots \ctau \csigma}_{\csigma \ctau\cdots \csigma \ctau} .  
\end{align}
for $m$ odd or even, respectively.  
We have the Jones--Wenzl relations 
\begin{align} 
({\sf 1}^{m-1}_{\ctau\csigma} \otimes {\sf spot}_\ctau^\emptyset) {\sf braid}^{\ctau\csigma}_{  \csigma\ctau}(m)
&=
 ({\sf spot}^\emptyset_\csigma \otimes  {\sf 1}_{ \ctau\csigma   }^{m-1})
{\sf JW}_{\csigma\ctau}
\\
({\sf 1}^{m-1}_{\ctau\csigma} \otimes {\sf spot}_\csigma^\emptyset) {\sf braid}^{\ctau\csigma}_{  \csigma\ctau}(m)
&=
 ({\sf spot}^\emptyset_\csigma \otimes {\sf 1}_{ \ctau\csigma   }^{m-1})
{\sf JW}_{\csigma\ctau} 
\intertext{for $m$ odd or even, respectively.  
 For   $(\csigma, \ctau ,\crho)\in S^3$ 
with   $m_{\csigma \crho}=
m_{  \crho\ctau}=2$ and 
$m_{\csigma \ctau}=m$,  we have }
({\sf braid}_{\csigma\ctau}^{\ctau\csigma}(m) \otimes {\sf 1}_
\crho)
{\sf braid}^{\csigma \ctau \cdots \csigma  \crho}_{\crho\csigma \ctau \cdots \csigma}
&=
{\sf braid}^{\ctau \csigma\cdots \ctau \crho}_{\crho\ctau \csigma\cdots \ctau}
({\sf 1}_
\crho\otimes {\sf braid}_{\csigma\ctau}^{\ctau\csigma}(m) ).
\end{align}
We have the  three  Zamolodchikov relations:    for a type $A_3$ triple $\al,\bet, \crho\in S $ with $m_{\al\bet}=3=m_{\al\gam}$ and $m_{\bet\gam}=2$ we have  that 
  \begin{align}\begin{split}\label{Zamolodchikov relations}
  &{\sf braid}^{ \gam\al\gam \bet\al\gam}_{ \al\gam\al  \bet\al\gam    }    
  {\sf braid}^{ \al\gam  \al\bet\al \gam    }_{  \al\gam \bet\al\bet    \gam   }    
   {\sf braid}^{ \al\gam\bet \al \bet\gam  }_{ \al\bet\gam \al  \gam\bet   }  
   {\sf braid}^{   \al\bet	\gam\al\gam \bet   }_{ \al\bet \al\gam\al   \bet }   
 {\sf braid}^{  \al\bet	 \al  \gam\al \bet  } _{  \bet \al\bet  \gam\al \bet    }   
{\sf braid}^{\bet\al \bet\gam\al\bet     }_{ \bet\al \gam\bet \al\bet     } 
\\[5pt] = \;\;\;&  	{\sf braid}_{\gam\al\bet\gam \al\gam  }^{ \gam\al \gam\bet\al\gam  }  
  {\sf braid}^{\gam\al\bet \gam\al \gam }_{\gam\al\bet\al\gam\al} 
   {\sf braid}^{\gam\al\bet\al\gam\al}_{\gam\bet\al\bet\gam\al}
 {\sf braid}^{\gam\bet\al\bet\gam\al}_{\bet\gam\al\gam\bet\al}   
 {\sf braid} ^{\bet\gam\al\gam\bet\al} 
 _{\bet \al\gam\al\bet\al} 
  {\sf braid}^{\bet\al\gam\al\bet\al}_{\bet\al\gam\bet\al\bet} .
  \end{split}
\intertext{ For a  type $B_3$   triple $\al,\bet,\gam\in S$ such that 
$m_{\al\gam}=4$, $m_{\bet\gam}=2$, $m_{\al\bet}=3$, we have that 
}\begin{split}
  &  {\sf braid}_{\gam\al\gam {\al\bet\al}{\gam\al\bet}}^{\gam\al\gam{\bet\al\bet} {\gam\al\bet}} 
    {\sf braid}^{\gam\al\gam{\al\bet\al}{\gam\al\bet}}
    _ { \al\gam{\al\gam\bet\al}{\gam\al\bet}}
     {\sf braid} 
    ^ { \al\gam{\al\gam\bet\al}{\gam\al\bet}}
    _ { \al\gam \al\bet \gam \al \gam\al\bet} 
    {\sf braid}^{ \al\gam \al\bet \gam \al \gam\al\bet} _
    { \al\gam \al\bet  \al \gam\al \gam\bet} 
    {\sf braid}^ 
    { \al\gam \al\bet  \al \gam\al \gam\bet} 
_    { \al\gam \bet  \al \bet  \gam\al  \gam \bet} 
\\[4pt]  
& {\sf braid}^ 
     { \al\gam \bet  \al \bet  \gam\al   \gam\bet} _
  { \al  \bet \gam \al    \gam \bet  \al  \bet\gam} \times
   {\sf braid}^ 
  { \al  \bet \gam \al    \gam \bet  \al  \bet\gam} _
    { \al  \bet \gam \al    \gam   \al  \bet  \al \gam} 
 {\sf braid}^ 
     { \al  \bet \gam \al    \gam   \al  \bet  \al \gam} 
     _
         { \al  \bet   \al    \gam   \al \gam \bet  \al \gam} 
 {\sf braid}^ 
         { \al  \bet   \al    \gam   \al \gam \bet  \al \gam} 
         _
  {   \bet   \al  \bet  \gam   \al\gam\bet     \al \gam}   
 {\sf braid}^ 
      {   \bet   \al  \bet  \gam   \al\gam\bet     \al \gam}   
      _
  {   \bet   \al     \gam\bet   \al \bet  \gam   \al \gam}         
\\[7pt]   
   =\;\;\;& {\sf braid}  
    ^
      {      \gam   \al \gam  \bet	 \al   	 \bet 	 \gam    \al  \bet   }   
_
  {      \gam   \al \gam  \bet	 \al 	 \gam   	 \bet   \al  \bet   }   
    {\sf braid}  
    ^
    {      \gam   \al \gam  \bet	 \al 	 \gam   	 \bet   \al  \bet   }   
    _
        {      \gam   \al    \bet	\gam \al 	 \gam   	  \al  \bet    \al }   
  {\sf braid}  
    ^
         {      \gam   \al    \bet	\gam \al 	 \gam   	  \al  \bet    \al }   
         _
                 {      \gam   \al    \bet		  \al 	 \gam   	  \al 	\gam			 \bet    \al }   
  {\sf braid}  
    ^
                          {      \gam   \al    \bet		  \al 	 \gam   	  \al 	\gam			 \bet    \al }   
                          _
                 {      \gam      \bet		  \al    \bet	 \gam   	  \al 	\gam			 \bet    \al }                             
  {\sf braid}  
    ^
                {      \gam      \bet		  \al    \bet	 \gam   	  \al 	\gam			 \bet    \al }                             
_
                 {          \bet	 \gam  	  \al    \gam   \bet	  	  \al 	 		 \bet \gam	   \al }                                                 
 \\[4pt]     %
 & {\sf braid}  
    ^
    {          \bet	 \gam  	  \al    \gam   \bet	  	  \al 	 		 \bet \gam	   \al } 
_
    {          \bet	 \gam  	  \al    \gam     \al   \bet     \al 	 \gam	   \al }     \times
  {\sf braid}  
    ^
      {          \bet	 \gam  	  \al    \gam     		\al   \bet     \al 	 \gam	   \al }     
      _
          {    \bet	   \al    \gam     \al  \gam  	 	 \bet     \al 	 \gam	   \al }     
 {\sf braid}  
    ^
          {    \bet	   \al    \gam     \al  \gam  	 	 \bet     \al 	 \gam	   \al }     
          _
    {    \bet	   \al    \gam     \al   	 \bet   \gam  	   \al 	 \gam	   \al }               
 {\sf braid}  
    ^
         {    \bet	   \al    \gam     \al   	 \bet   \gam  	   \al 	 \gam	   \al }               
_
    {    \bet	   \al    \gam     \al   	 \bet     	   \al 	 \gam	   \al \gam 		}                        
 {\sf braid}  
    ^
        {    \bet	   \al    \gam     \al   	 \bet     	   \al 	 \gam	   \al \gam 		}      
_      {    \bet	   \al    \gam      	 \bet     	   \al 	 	 \bet      \gam	   \al \gam 		}      
       \end{split}                        
\end{align}
and for a type   $H_3$  
 triple $\al,\bet,\gam\in S$ such that 
$m_{\al\gam}=2$, $m_{\bet\gam}=5$, $m_{\al\bet}=3$, we have 
a final $H_3$  
 relation\footnote{To the authors' knowledge, this relation has not been explicitly determined (but  can be given more computing power). We invite the reader to either believe that this can be written down (as is now standard in this area) or to read all results in this paper ``modulo" any Coxeter group $W$ with  a parabolic subgroup of type~$H_3$.}, for which  we   refer to \cite[Definition 5.2]{MR3555156}. 
 Further,  we require the bifunctoriality relation
\begin{align}
 \big(({\sf D  }_1\circ	{\sf 1}_{ {\w}}   )\otimes  ({\sf D}_2 \circ {\sf 1}_{ {\x}}) \big)
\big(({\sf 1}_{\w} \circ {\sf D}_3)  \otimes ({\sf 1}_\x \circ{\sf D }_4)\big)
=
 ({\sf D}_1 \circ {\sf 1}_{{\w}} \circ {\sf D_3}) \otimes ({\sf D}_2 \circ {\sf 1}_{{\x}} \circ {\sf D}_4)
\end{align}
and the monoidal unit relation
\begin{align}
{\sf 1}_{\emptyset} \otimes {\sf D}_1={\sf D}_1={\sf D}_1 \otimes {\sf 1}_{\emptyset}
\end{align}
for all diagrams ${\sf D}_1,{\sf D}_2,{\sf D}_3,{\sf D}_4$ and all  words $\w,\x$.  
Finally, we require the (non-local) cyclotomic relation
\begin{align}
{\sf spot}^\emptyset_\al {\sf spot}^\al_\emptyset \otimes {\sf 1}_\vvv =0 \qquad \text{for  all  $\vvv\in {\sf exp}(w)$, $w\in W,$  and all $\al \in S$.  }
\end{align}

\renewcommand{\w}{{\underline{w}}}
\renewcommand{\vvv}{{\underline{z} }}
\renewcommand{\y}{{\underline{y}}}
\renewcommand{\x}{{\underline{x}}}

\noindent
The $\mathbb Z$-grading on the algebra $\mathcal{H}^\Bbbk_{  W}$ is defined on the generators (and their duals) as follows: 
\begin{align}
 {\sf deg}({\sf 1}_\emptyset)=0
 \quad
  {\sf deg}({\sf 1}_\al)=0
  \quad
  {\sf deg} ({\sf spot}^\emptyset_\al)=1
    \quad
  {\sf deg} ({\sf fork}^\al_{\al\al})=-1
    \quad
  {\sf deg} ({\sf braid}^{\al\bet}_{\bet\al}(m))=0
\end{align}
for $\al,\bet \in S$   arbitrary and $m\geq 2$.

 \end{defn}

\begin{rmk}
The   cyclotomic  relation  amounts to considering diagrammatic Soergel modules instead of diagrammatic Soergel bimodules, or   equivalently, to considering finite dimensional $\Bbbk$-modules rather than modules of finite rank over the polynomial ring, $R$, generated by the ``barbells", ${\sf spot}^\emptyset_\csigma {\sf spot}^\csigma_\emptyset$, for $\csigma \in S$.
If we omit the cyclotomic relation in the above definition of $\mathcal{H}_W$ we obtain a diagrammatic Bott--Samelson category $\mathcal{H}_{\rm BS}$ for $W$ (viewed as a locally unital associative $\Bbbk$-algebra).

Diagrammatic Bott--Samelson categories are normally defined using a reflection representation $\mathfrak{h}=(V, \{\alpha_{\csigma}^\vee : \csigma \in S\}, \{\alpha_{\csigma} : \csigma \in S\})$ of the Coxeter group $W$ called a {\sf realisation}. 
Our construction of $\mathcal{H}_{\rm BS}$ implicitly uses the {\sf universal realisation} of $W$ with respect to the balanced Cartan matrix $(a_{\csigma \ctau})_{\csigma,\ctau \in S}$, defined as follows. 
Abusing notation slightly, let $V^\ast$ be a free $\Bbbk$-module with basis
$\{\alpha_\csigma : \csigma \in S\}$, and let $V=(V^\ast)^\ast$.
For each $\csigma \in S$ define $\alpha_{\csigma}^\vee \in V$ by setting $\langle \alpha_{\csigma}^\vee, \alpha_{\ctau} \rangle = a_{\csigma \ctau}$ for all $\ctau \in S$.
The Coxeter group $W$ acts on $V^\ast$ via $\csigma(\beta)=\beta-\langle \alpha_{\csigma}^\vee, \beta\rangle \alpha_{\csigma}$ for all $\csigma \in S$ and $\beta \in V^\ast$.
If $\mathcal{H}_{\rm BS}'$ is the Bott--Samelson category for another realisation of $W$ with the same Cartan matrix, then there is a unique monoidal functor $\mathcal{H}_{\rm BS} \rightarrow \mathcal{H}_{\rm BS}'$, which descends to an isomorphism after taking cyclotomic quotients (cf.~\cite[Lemma~11.2]{MR3805034}).
\end{rmk} 
 
         \begin{figure}[ht!]
\begin{align*} 
 &\begin{minipage}{1.6cm}\begin{tikzpicture}[scale=0.75]
\draw[densely dotted](0,0)  circle (30pt);
\clip(0,0) circle (30pt);
\draw[line width=0.06cm, magenta, fill=magenta](0,0)--++(90:3.5) 
  circle (4pt);
\draw[line width=0.06cm, magenta, fill=magenta](0,0)--++(-30:3.5) 
  circle (4pt); 
  \draw[line width=0.06cm, magenta, fill=magenta](0,0)--++(-150:3.5) 
  circle (4pt);
  \draw[line width=0.06cm, cyan, fill=cyan](0,0)--++(-90:3.5) 
  circle (4pt);
    \draw[line width=0.06cm, cyan, fill=cyan](0,0)--++(150:3.5) 
  circle (4pt);
    \draw[line width=0.06cm, cyan, fill=cyan](0,0)--++(30:0.6) coordinate(hi)
 ;
    \draw[line width=0.06cm, cyan, fill=cyan] (hi)--++(90:1);
        \draw[line width=0.06cm, cyan, fill=cyan] (hi)--++(-30:1);
 ;
\end{tikzpicture}
\end{minipage}\;=\;
 \begin{minipage}{1.6cm}\begin{tikzpicture}[scale=0.75]
\draw[densely dotted](0,0)  circle (30pt);
\clip(0,0) circle (30pt);
\draw[line width=0.06cm, magenta, fill=magenta](0,0)--++(90:3.5) 
  circle (4pt);
\draw[line width=0.06cm, magenta, fill=magenta](0,0)--++(-30:3.5) 
  circle (4pt); 
  \draw[line width=0.06cm, magenta, fill=magenta](0,0)--++(-150:3.5) 
  circle (4pt);
   \draw[line width=0.06cm, magenta, fill=white](0,0.6)--(-30:0.6)
  coordinate (whip1)--(-150:0.6)  coordinate (whip2)--(0,0.6);   
      \draw[line width=0.06cm, magenta, fill=white](0,0.6) to [out=-45, in=105] (whip1)
    to [out=-165, in=-15] (whip2)to [out=75, in=-135] (0,0.6); 
  \path(0,0)--++(-90:3.5) 
  circle (4pt);
  \path(0,0)--++(150:3.5) 
  circle (4pt);
  \path(0,0)--++(30:0.6) coordinate(hi)
 ;
  \path (hi)--++(90:0.6)coordinate(hi3);
      \path (hi)--++(-30:0.6)coordinate(hi2);
 ;
\draw[line width=0.06cm, cyan] (hi3)--(0,0.6);
\draw[line width=0.06cm, cyan] (hi2)--(-30:0.6) to[out=150,in=-90] (0,0.6)--(hi3);
\draw[line width=0.06cm, cyan] (-30:0.6)--(0,-31pt);
\path(0,0)--++(150:31pt) coordinate(hello2);
\draw[line width=0.06cm, cyan] (0,0.6)--(hello2);
 \end{tikzpicture}
\end{minipage}
   &
 \begin{minipage}{1.6cm}\begin{tikzpicture}[xscale=0.75,yscale=-0.75]
\draw[densely dotted](0,0)  circle (30pt);
\clip(0,0) circle (30pt);
\draw[line width=0.06cm, cyan, fill=cyan](0,0)--++(90:3.5) 
  circle (4pt);
\draw[line width=0.06cm, cyan, fill=cyan](0,0)--++(-30:3.5) 
  circle (4pt); 
  \draw[line width=0.06cm, cyan, fill=cyan](0,0)--++(-150:0.7) 
  circle (4pt);
    \draw[line width=0.06cm, magenta, fill=magenta](0,0)--++(-90:3.5) 
  circle (4pt);
    \draw[line width=0.06cm, magenta, fill=magenta](0,0)--++(150:3.5) 
  circle (4pt);
    \draw[line width=0.06cm, magenta, fill=magenta](0,0)--++(30:3.5) coordinate(hi)
 ;
    \draw[line width=0.06cm, magenta, fill=magenta] (hi)--++(90:1);
        \draw[line width=0.06cm, magenta, fill=magenta] (hi)--++(-30:1);
 ;
 \end{tikzpicture}
\end{minipage}\;=\;
 \begin{minipage}{1.6cm}\begin{tikzpicture}[xscale=0.75,yscale=-0.75]\draw[densely dotted](0,0)  circle (30pt);
\clip(0,0) circle (30pt);
\draw[line width=0.06cm, cyan, fill=cyan]
(0,0.6)  circle (4pt)  --++(90:3.5) 
 ;
\draw[line width=0.06cm, cyan, fill=cyan]
(-30:0.6)   circle (4pt) --++(-30:3.5) 
; 
   \draw[line width=0.06cm, magenta, fill=magenta](0,0)--++(-90:3.5) 
  circle (4pt);
    \draw[line width=0.06cm, magenta, fill=magenta](0,0)--++(150:3.5) 
  circle (4pt);
    \draw[line width=0.06cm, magenta, fill=magenta](0,0)--++(30:3.5) coordinate(hi)
 ;
    \draw[line width=0.06cm, magenta, fill=magenta] (hi)--++(90:1);
        \draw[line width=0.06cm, magenta, fill=magenta] (hi)--++(-30:1);
 ;
\end{tikzpicture}\end{minipage}\; +
\;
  \begin{minipage}{1.6cm}\begin{tikzpicture}[xscale=0.75,yscale=-0.75]
\draw[densely dotted](0,0)  circle (30pt);
\clip(0,0) circle (30pt);
\path(0,0)--++(90:30pt) coordinate(N);
\path(0,0)--++(30:30pt) coordinate(NE);
\path(0,0)--++(-30:30pt) coordinate(SE);
\path(0,0)--++(-90:30pt) coordinate(S); 
\path(0,0)--++(-150:30pt) coordinate(SW); 
\path(0,0)--++(150:30pt) coordinate(NW);  
\draw[line width=0.06cm, magenta] (S) to [out=80, in=-20]  (NW);
\draw[line width=0.06cm, cyan] (SE) to [out=160, in=-100]  (N);
\draw[line width=0.06cm, magenta,fill=magenta] (NE)--++(-150:0.6) circle (4pt);
 \end{tikzpicture}
\end{minipage}
\end{align*}
\caption{The  fork-braid and Jones--Wenzl relations for $m_{\csigma\ctau}=3$. }
    \label{sahkdgkshjagdjkshfgsdjhk}
   \end{figure}

\begin{figure}[ht!] 
 \begin{minipage}{5cm}\begin{tikzpicture} [scale=0.75]
\draw[densely dotted, rounded corners=13mm](-2,1) rectangle (4,-5.5); 
 \clip[rounded corners=13mm](-2,1) rectangle (4,-5.5); 
\path(0,0) coordinate (origin);
\path(origin)--++(-30:2) coordinate (origin1);
\path(origin1)--++(-90:2) coordinate (origin2);
\path(origin2)--++(-150:2) coordinate (origin3);

\path(origin2)--++(-90:2) coordinate (originfake2);
\path(origin3)--++(-90:2) coordinate (originfake3);

\path(originfake2)--++(-30:1) 
 ;
  \path(originfake2)--++(-150:1) 
coordinate (hi3); 
      \path(originfake2)--++(150:1) coordinate (BLUE) 
;
   \path(originfake2)--++(-90:1) 
coordinate (hi2);
    \path(originfake2)--++(-150:1)  
 ;
    \path(originfake2)--++(30:1) coordinate (BLUE2) ;
 \path(originfake3)--++(-30:1) 
 ;
  \path(originfake3)--++(-150:1) 
coordinate (hi3); 
      \path(originfake3)--++(90:1) 
;
   \path(originfake3)--++(-90:1) 
coordinate (hi2);
    \path(originfake3)--++(150:1)
 ;
    \path(originfake3)--++(30:1) coordinate (BLUE3) ;

\draw[line width=0.08cm, magenta, magenta](origin)--++(-30:1) 
 ;
  \draw[line width=0.08cm, magenta, magenta](origin)--++(-150:1) 
coordinate (hi3); 
      \draw[line width=0.08cm, magenta, magenta](origin)--++(90:1) 
;
   \draw[line width=0.08cm, darkgreen, darkgreen](origin)--++(-90:1) 
coordinate (hi2);
    \draw[line width=0.08cm, darkgreen, darkgreen,rounded corners=2mm](origin)--++(150:1)--++(90:1)
;
    \draw[line width=0.08cm, darkgreen, darkgreen,rounded corners=2mm](origin)--++(30:1) --++(90:172);

\draw[line width=0.08cm, cyan, cyan](origin1)--++(-30:1) 
coordinate (hi4); 
\path(hi4)--++(30:1) coordinate (hi5);
  \draw[line width=0.08cm, cyan, cyan](origin1)--++(-150:1) 
coordinate(hi);
      \draw[line width=0.08cm, cyan, cyan,rounded corners](origin1)--++(90:3) 
;
   \draw[line width=0.08cm, magenta, magenta](origin1)--++(-90:1) 
;
    \draw[line width=0.08cm, magenta, magenta](origin1)--++(150:1) 
;
    \draw[line width=0.08cm, magenta, magenta,rounded corners=2mm](origin1)--++(30:1) --++(90:2);

\draw[line width=0.08cm, magenta, magenta,rounded corners=2mm](origin2)--++(-30:1) --++(-90:2)
; 
  \draw[line width=0.08cm, magenta, magenta](origin2)--++(-150:1) 
  
;
      \draw[line width=0.08cm, magenta, magenta](origin2)--++(90:1) 
;
   \draw[line width=0.08cm, darkgreen, darkgreen](origin2)--++(-90:1) to [out=-90,in=90] (BLUE3); 
;
    \draw[line width=0.08cm, darkgreen, darkgreen](origin2)--++(150:1) 
    to [out=150,in=-90] (hi2)
;
    \draw[line width=0.08cm, darkgreen, darkgreen](origin2)--++(30:1)
    coordinate (hi6)
    to [out=30,in=-90] (hi5) --++(90:3);

\path(hi6)--++(-30:1) coordinate (hi7);


\draw[line width=0.08cm, cyan, cyan]  (hi4) to [out=-30, in=90] (hi7)--++(-90:2.5); 

   \draw[line width=0.08cm, cyan, cyan,rounded corners =2mm](origin3)--++(-150:1) --++(-90:2)
;

\draw[line width=0.08cm, cyan, cyan,rounded corners =2mm](origin3)--++(-30:1) --(originfake2)  --++(-90:2)
;

      \draw[line width=0.08cm, cyan ](origin3)--++(90:1)to [out =90,in=-150] (hi)
;
   \draw[line width=0.08cm, magenta,  magenta](origin3)--++(-90:2) 
;
    \draw[line width=0.08cm, magenta, magenta](origin3)--++(150:1) 
    to [out=150, in=-150] (hi3) 
;
    \draw[line width=0.08cm, magenta, magenta](origin3)--++(30:1) ;

 \end{tikzpicture}
 \end{minipage}
\!\!\!\!=\;\;
\begin{minipage}{5cm}\begin{tikzpicture} [scale=-0.75]
\draw[densely dotted, rounded corners=13mm](-2,1) rectangle (4,-5.5); 
 \clip[rounded corners=13mm](-2,1) rectangle (4,-5.5); 
\path(0,0) coordinate (origin);
\path(origin)--++(-30:2) coordinate (origin1);
\path(origin1)--++(-90:2) coordinate (origin2);
\path(origin2)--++(-150:2) coordinate (origin3);

\path(origin2)--++(-90:2) coordinate (originfake2);
\path(origin3)--++(-90:2) coordinate (originfake3);

\path(originfake2)--++(-30:1) 
 ;
  \path(originfake2)--++(-150:1) 
coordinate (hi3); 
      \path(originfake2)--++(150:1) coordinate (BLUE) 
;
   \path(originfake2)--++(-90:1) 
coordinate (hi2);
    \path(originfake2)--++(-150:1)  
 ;
    \path(originfake2)--++(30:1) coordinate (BLUE2) ;
 \path(originfake3)--++(-30:1) 
 ;
  \path(originfake3)--++(-150:1) 
coordinate (hi3); 
      \path(originfake3)--++(90:1) 
;
   \path(originfake3)--++(-90:1) 
coordinate (hi2);
    \path(originfake3)--++(150:1)
 ;
    \path(originfake3)--++(30:1) coordinate (BLUE3) ;

\draw[line width=0.08cm, magenta, magenta](origin)--++(-30:1) 
 ;
  \draw[line width=0.08cm, magenta, magenta](origin)--++(-150:1) 
coordinate (hi3); 
      \draw[line width=0.08cm, magenta, magenta](origin)--++(90:1) 
;
   \draw[line width=0.08cm, cyan, cyan](origin)--++(-90:1) 
coordinate (hi2);
    \draw[line width=0.08cm, cyan, cyan,rounded corners=2mm](origin)--++(150:1)--++(90:1)
;
    \draw[line width=0.08cm, cyan, cyan,rounded corners=2mm](origin)--++(30:1) --++(90:172);

\draw[line width=0.08cm, darkgreen, darkgreen](origin1)--++(-30:1) 
coordinate (hi4); 
\path(hi4)--++(30:1) coordinate (hi5);
  \draw[line width=0.08cm, darkgreen, darkgreen](origin1)--++(-150:1) 
coordinate(hi);
      \draw[line width=0.08cm, darkgreen, darkgreen,rounded corners](origin1)--++(90:3) 
;
   \draw[line width=0.08cm, magenta, magenta](origin1)--++(-90:1) 
;
    \draw[line width=0.08cm, magenta, magenta](origin1)--++(150:1) 
;
    \draw[line width=0.08cm, magenta, magenta,rounded corners=2mm](origin1)--++(30:1) --++(90:2);

\draw[line width=0.08cm, magenta, magenta,rounded corners=2mm](origin2)--++(-30:1) --++(-90:2)
; 
  \draw[line width=0.08cm, magenta, magenta](origin2)--++(-150:1) 
  
;
      \draw[line width=0.08cm, magenta, magenta](origin2)--++(90:1) 
;
   \draw[line width=0.08cm, cyan, cyan](origin2)--++(-90:1) to [out=-90,in=90] (BLUE3); 
;
    \draw[line width=0.08cm, cyan, cyan](origin2)--++(150:1) 
    to [out=150,in=-90] (hi2)
;
    \draw[line width=0.08cm, cyan, cyan](origin2)--++(30:1)
    coordinate (hi6)
    to [out=30,in=-90] (hi5) --++(90:3);

\path(hi6)--++(-30:1) coordinate (hi7);


\draw[line width=0.08cm, darkgreen, darkgreen]  (hi4) to [out=-30, in=90] (hi7)--++(-90:2.5); 

   \draw[line width=0.08cm, darkgreen, darkgreen,rounded corners =2mm](origin3)--++(-150:1) --++(-90:2)
;

\draw[line width=0.08cm, darkgreen, darkgreen,rounded corners =2mm](origin3)--++(-30:1) --(originfake2)  --++(-90:2)
;

      \draw[line width=0.08cm, darkgreen ](origin3)--++(90:1)to [out =90,in=-150] (hi)
;
   \draw[line width=0.08cm, magenta,  magenta](origin3)--++(-90:2) 
;
    \draw[line width=0.08cm, magenta, magenta](origin3)--++(150:1) 
    to [out=150, in=-150] (hi3) 
;
    \draw[line width=0.08cm, magenta, magenta](origin3)--++(30:1) ;

 \end{tikzpicture}
 \end{minipage}
 \caption{The Zamolodchikov relation  for $A_3$.  } 
 \label{Katya}
\end{figure}

\begin{defn}
Given $S_P\subseteq S$       we define  the {\sf anti-spherical Hecke category}
$\mathcal{H}^\Bbbk_{P \backslash W} $ to be the quotient of $\mathcal{H}^\Bbbk_W$ by the homogeneous (non-local) relation
\begin{align}\label{annihilation}
{\sf 1}_\csigma \otimes {\sf 1}_\w=0
\end{align}
for all $\csigma \in S_P\subseteq S$ and $\w \in {\sf exp}(w)$ for $w\in W$.  
\end{defn}

\subsection{Parabolic light leaves tableaux  and  cellular bases 
} \label{lighteaves}
We now recall the combinatorics of cellular bases for diagrammatic Hecke categories. \color{black}
This is well known in the non-parabolic setting (see e.g.~\cite[\S 6.1]{MR3555156} or \cite[Chapter~10.4]{EMTW}); a good reference for this material in the parabolic setting is \cite[\S 5]{antiLW}. 
Our notation is closely analogous to that in previous work of the first and second authors (\cite{pathbasesBCHM} and \cite{cell4us2}) and Ryom--Hansen \cite{MR4057514}. 
In particular, we will use the language of tableaux (rather than words in the Coxeter generators) to describe the indexing sets for our cellular bases.  
We provide   extended examples after the definitions  see \cref{heresmyexample,heresmyexample2} 
(in particular, we highlight the role played by the parabolic in these examples).  
  \color{black}

We will consider certain truncations of $\mathcal{H}^\Bbbk_{P\backslash W}$, and 
to that end we define, for any  $w \in {^P}W$, a poset 
\begin{align}\mathcal{P}_{\leq  w} = \{  x \mid   x\in {^PW} \text { and }x\leq w\} 
\end{align}
 partially ordered by the  Bruhat order.  
  Fix $\w=\sigma_1\dots \sigma_\ell \in {\sf exp}_P^\ell(w)$  (not necessarily reduced).
Given $\stt$ a subword of $\w\in{\sf exp}_P^\ell(w)$,  we define  
$\Shape_{k}(\stt)=\sigma_1 ^{t_1} \sigma_2^{t_2} \dots \sigma_k^{t_k} \in W  $  for $1\leq k\leq \ell $.      
    In the non-parabolic case, the set of  {\sf tableaux of shape $ x$ and weight $\w$} will then be given by  
\begin{align}
   \Std_{\leq \w}(  x)=\{ {\sf t} \mid   \Shape_{\ell}(\stt) =x    \}      
\end{align}
 and we define  the set of {\sf parabolic   tableaux of shape $ x$ and weight $\w$} to be
\begin{align}\label{pfpfpfpfpfpf}
   \Std^P_{\leq \w}(  x)=\{ {\sf t} \mid   \Shape_{k}(\stt) \sigma_{k+1} \in {^PW} \text{ for }0\leq k <\ell \text{ and }  \Shape_{\ell}(\stt) = x\}\subseteq   \Std _{\leq \w}(  x).     
\end{align}
Finally, we take the union over all possible {\sf weights} to obtain the set of all parabolic tableaux   
\begin{align}
    \Std^P  _{\leq \ell}( x)= 
   \bigcup_{\begin{subarray}c w\in {^PW}
  \\
  \w \in  { \sf exp}_P^\ell(w)  
  \end{subarray}} \Std^P_{\leq \w}(  x) \text{,}
 \qquad 
  \Std^P   ( x)=  \bigcup_{\ell\geq 0}\Std^P  _{\leq \ell}( x) . 
  \end{align}
      Given $x<x\ctau$ and  $\stt \in \Std^P_{\leq \w }( x )$, we define 
\begin{align}\label{plusminus}
{\stt^+}=(t_1,\dots,t_q,1)\in \Std_{\leq \w\ctau}( x\ctau)\text{,}\qquad 
    {\stt^-}=(t_1,\dots,t_q,0)\in \Std_{\leq \w\ctau}( x ) 
    \end{align}
    and this will be the backbone of how we grow the cellular bases. 
We can decompose the diagrammatic anti-spherical Hecke category 
 in the following manner,
\begin{align}\mathcal{H}_{P\backslash W}^\Bbbk=\bigoplus _{
\begin{subarray} c 
\uvee \in   {\sf exp}_P(v) 
\\
\w \in   {\sf exp}_P(w)
\\
v,w \in {   W}
\end{subarray}
}{\sf 1}_\uvee \mathcal{H}_{P\backslash W}^\Bbbk{\sf 1}_\w
\end{align}
and hence regard this algebra as a locally unital associative algebra in the sense of 
\cite[Section 2.2]{BSBS}.  
Each one of these  finite-dimensional pieces  
 ${\sf 1}_\uvee \mathcal{H}_{P\backslash W}^\Bbbk{\sf 1}_\w$ is 
 the space of  morphisms between the Bott--Samelson objects labelled by $\uvee $ and $\w$.  
%
%
%
%
%
%

  Recall that given  $\w=\sigma_1\sigma_2\cdots  \sigma_\ell$  and 
${\sf t}=(t_1,t_2,\dots ,t_\ell)\in\{0,1\}^\ell$, we set 
 $\w^\stt  :=\sigma_1^{t_1}\sigma_2^{t_2}\cdots \sigma_\ell^{t_\ell}$.  
We define 
\begin{align}
    {\sf 1}_{\leq \w}
 =\sum_{\stt \in \{0,1\}^\ell }{\sf 1}_{\w^\stt } .
\end{align}
 \color{black}
In particular, we will wish to consider the   subalgebras
\begin{align} {\sf 1}_{ \leq \w }\mathcal{H}_{P\backslash W}^\Bbbk {\sf 1}_{ \leq \w}  
\end{align}
 in order to understand $\mathcal{H}_{P\backslash W}^\Bbbk$.  
 We  will   construct a spanning set of 
     $ {\sf 1}_{ \leq \w }\mathcal{H}_{P\backslash W}^\Bbbk {\sf 1}_{ \leq \w } $  %
in an iterative fashion. 
    For any fixed expression $\w \in {\sf exp}_P
   (w)$, we have an embedding 
 \begin{equation}\label{preimage}{\sf 1}_{\leq \w  } \mathcal{H}_{P\backslash W}^\Bbbk{\sf 1}_{\leq \w } \hookrightarrow 
({\sf 1}_{\leq \w} \otimes {\sf 1}_{\ctau}) \mathcal{H}_{P \backslash W}^\Bbbk ({\sf 1}_{\leq \w} \otimes {\sf 1}_{\ctau}) 
\end{equation} given by $D \mapsto D\otimes {\sf 1}_\ctau $.
Note that the image of this embedding lies inside an idempotent truncation of ${\sf 1}_{\leq \w\ctau } \mathcal{H}_{P\backslash W}^\Bbbk{\sf 1}_{\leq \w\ctau }$.
\color{black}
We consider each of these embeddings in turn (for all $\ctau \in S$) in order to provide the iterative construction of the ``anti-spherical light leaves" elements of $\mathcal{H}_{P\backslash W}^\Bbbk$.  

 We now inductively construct the light leaves basis. 
   For   $\stt \in \Std_{\leq 1_W}(1_W)$, we formally set $c_\stt=1_\emptyset$ and we recall that $\stt^+$ and $\stt^-$ are defined in \cref{plusminus}.   
 \color{black} 
 If   $  {y}\ctau    >  y  $, then for any $ \y\in {\sf rexp}_P
 (y)$, $\y^+\in {\sf rexp}_P
 (y\ctau)$, 
  $\y^-\in {\sf rexp}_P
  (y)$  and  $\stt \in \Std_{\leq  \w}(y)$  
 we define 
\begin{align}
c_{{\stt^+}}= {\sf braid}^{\y^+}_{\y\ctau} (c_{\stt } \otimes {\sf 1}_\ctau)\qquad 
c_{{\stt^-}}=
{\sf braid}^{\y^- }_{\y  } ( c_{\stt } \otimes {\sf spot}_\ctau^\emptyset  ).
\end{align}
 If $ {y} \ctau=y'< y   $, then  we let 
   $\y' \ctau$ be a reduced expression for $y$.  
   For any  $\y^+\in {\sf rexp}_P
    (y  )$, 
  $\y^-\in {\sf rexp}_P
  (y\ctau)$  and  $\stt \in \Std_{\leq  \w
   }(y'\ctau
    )$  
 we define 
\begin{align}
c_{{\stt^+}}= 
{\sf braid}^{\y^+}_{\y'  \ctau }
 ({\sf1}_{\y' }\otimes   {\sf fork}^{\ctau }_{\ctau\ctau})
 (c_{\stt } \otimes {\sf 1}_\ctau) 
\qquad 
c_{{\stt^-}}= 
{\sf braid}^{\y^-}_{\y'  }   ({\sf1}_{\y' }\otimes  {\sf spot}_\ctau^\emptyset   {\sf fork}^{\ctau }_{\ctau\ctau})
 (c_{\stt } \otimes {\sf 1}_\ctau). 
 \end{align}
 Fix a choice of reduced expression $\x$ for $x\in \mathcal{P}_{\leq  w}$  and construct elements $c_\sts, c_\stt$ for 
$\sts,\stt \in  \Std^P_{\leq w}( x)$. We set  $   c_{\sts\stt}:= (c_\sts ^\ast) c_\stt$.  

The definition of the anti-spherical Hecke category $\mathcal{H}^\Bbbk_{P \backslash W}$ is extremely general, making sense over arbitrary rings and Coxeter systems.
In order for it to be well behaved we will make the following (very mild) assumption.

\begin{assump}\label{assump:antispherLL}
The anti-spherical light leaves elements
\begin{equation}\label{basis}
\{c_{\sts\stt} \mid \sts,\stt\in \Std^P_{\leq w}( x) ,   x \in \mathcal{P}_{\leq  w}  \}
\end{equation}
are $\Bbbk$-linearly independent in ${\sf 1}_{\leq \w}\mathcal{H}^\Bbbk_{P \backslash W} {\sf 1}_{\leq \w}$.
\end{assump}

\begin{thm}[{\cite[Section 6.4]{MR3555156} and {\cite[Theorem 5.3]{antiLW}}}] \label{LEW1}
If \Cref{assump:antispherLL} holds, then the algebra 
${\sf 1}_{\leq \w}\mathcal{H}^\Bbbk_{P \backslash W} {\sf 1}_{\leq \w}$ is finite-dimensional  with graded cellular basis \eqref{basis}
with respect to the Bruhat ordering on 
$ \mathcal{P}_{\leq  w}$ and anti-involution $\ast$.  
For $\Bbbk$   a field, we have that ${\sf 1}_{\leq \w}\mathcal{H}^\Bbbk_{P \backslash W} {\sf 1}_{\leq \w}$ is quasi-hereditary.  
\end{thm}

\begin{proof}
In  the proof of {\cite[Theorem 5.3]{antiLW}} it is proven that 
\eqref{basis} always spans, so \Cref{assump:antispherLL} implies that it is in fact a basis.  
Cellularity  is not mentioned explicitly, but follows in a completely analogous fashion to {\cite[Section 6.4]{MR3555156}}.  
 The only point of the theorem which is not explicitly stated  in  {\cite[Section 6.4]{MR3555156}} and {\cite[Theorem 5.3]{antiLW}} 
 is that the algebra is quasi-hereditary.  However, this is immediate from the fact that each layer of the cellular basis contains (at least one) idempotent $c_{\sts\sts}={\sf 1}_\x$ for $\sts$ the unique tableau in $ \Std_{\leq \x}(x)\subseteq  \Std_{\leq \w}(x)$.  
 \end{proof}

\begin{rmk} We note that when \Cref{assump:antispherLL} does not hold, the analogue of Soergel's categorification theorem for $\mathcal{H}^\Bbbk_{P \backslash W}$ is false \cite[Theorem 6.2]{antiLW}.     
In this instance it is debatable whether $\mathcal{H}^\Bbbk_{P \backslash W}$ should even be called the ``anti-spherical Hecke category''!
The following result is our attempt to give a reasonably general condition for when \Cref{assump:antispherLL} holds.
\end{rmk}

\begin{thm}[{\cite[Theorem 5.3]{antiLW}}] \label{thm:univrealff}
Let $\mathcal{O}$ be a complete local ring in which $2$ is invertible, and let $(a_{\csigma \ctau}^{\mathcal{O}})_{\csigma,\ctau \in S}$ be a balanced Cartan matrix for $W$ over $\mathcal{O}$.
If the universal realisation for $W$ with respect to $(a_{\csigma \ctau}^{\mathcal{O}})_{\csigma,\ctau \in S}$ is faithful, then \Cref{assump:antispherLL} holds for $\mathcal{H}_{P \backslash W}^{\mathcal{O}}$, the anti-spherical Hecke  category defined over $\mathcal{O}$ with respect to $(a_{\csigma \ctau}^{\mathcal{O}})_{\csigma,\ctau \in S}$.
Moreover, if there is a ring homomorphism $\mathcal{O} \rightarrow \Bbbk$ such that $(a_{\csigma \ctau})_{\csigma,\ctau \in S}$ is the image in $\Bbbk$ of $(a_{\csigma \ctau}^{\mathcal{O}})_{\csigma,\ctau \in S}$, then \Cref{assump:antispherLL} holds for $\mathcal{H}^\Bbbk_{P \backslash W}$.
\end{thm}

\begin{proof}
  \Cref{assump:antispherLL} is preserved by base change, so the second statement follows immediately from the first.
The special case where $\mathcal{O}=\mathbb{R}$ and $(a_{\csigma \ctau}^{\mathcal{O}})_{\csigma,\ctau \in S}$ is the ``geometric'' Cartan matrix for $W$ over $\mathbb{R}$ is proved in \cite[Proposition 5.5]{antiLW}.
In fact this proof is valid for any universal realisation over a complete local ring in which the ``parabolic property'' \cite[2.3]{antiLW} holds, including faithful realisations.
\end{proof}

There are two incredibly important realisations upon which we now focus our attention.  

\begin{eg} \label{impottt}
Let $p$ be an odd prime.
\begin{enumerate}[leftmargin=*]
\item Let $\Bbbk$ be any field of characteristic $p$ and let $W$ be a crystallographic Coxeter group.  
Let $A$ be a generalised Cartan matrix for $W$, and let $(a_{\csigma\ctau})_{\csigma,\ctau \in S}$ be the image of $A$ over $\Bbbk$.
Set $\mathcal{O}=\mathbb{Z}_p$ take $(a_{\csigma \ctau}^{\mathcal{O}})_{\csigma,\ctau \in S}$ to be the image of $A$ over $\mathcal{O}$.
The universal realisation for $W$ with respect to $A$ is faithful (see e.g.~\cite[Chapter 3]{kac}), so the same holds true over $\mathcal{O}$.
Thus \Cref{assump:antispherLL} holds for $\mathcal{H}^\Bbbk_{P \backslash W}$ by \Cref{thm:univrealff}.

\item Let $\Bbbk$ be a field of characteristic $p$ which contains the algebraic integers $2\cos(\pi/m_{\csigma \ctau})$ for all $\csigma,\ctau \in S$, and define the ``geometric'' Cartan matrix by setting $a_{\csigma \ctau}=-2\cos(\pi/m_{\csigma \ctau})$.
As mentioned above, it is known that the universal realisation with respect to this Cartan matrix over $\mathbb{R}$ is faithful, so the same holds true over $\mathcal{O}=\mathbb{Z}_p[2\cos(\pi/m_{\csigma \ctau}) \mid \csigma,\ctau \in S]$. 
Thus \Cref{assump:antispherLL} holds for $\mathcal{H}^\Bbbk_{P \backslash W}$ by \Cref{thm:univrealff}.
\end{enumerate}
\end{eg}

\begin{rmk}\color{black}
Typically the realisations in \Cref{impottt} do \emph{not} satisfy the parabolic property over $\Bbbk$.
For example, suppose $W$ is an affine Weyl group, $(a_{\csigma\ctau})_{\csigma,\ctau \in S}$ is the (image in $\Bbbk$ of the) associated affine Cartan matrix, and $P$ is the finite Weyl subgroup. 
Then the corresponding realisation from \cref{impottt}(1) does not satisfy the parabolic property \cite[Lemma~1.11]{matrix-recursion}.  This is also discussed in \cite[Sections 2.3 and 3.1]{antiLW}.
\end{rmk}

\begin{rmk}
In particular, we note that if $\Bbbk$ is algebraically closed and of odd (or infinite) characteristic then \Cref{assump:antispherLL} holds for the geometric Cartan matrix (by \cref{{impottt}}(2)).
Thus the reader unfamiliar with realisations can focus on just this case.
\end{rmk}

\begin{rmk}
We have assumed that $p\neq 2$ as we do not wish to discuss the technicalities of Demazure surjectivity.  Demazure surjectivity sometimes fails for the natural and geometric realisations in characteristic 2 (even for crystallographic types). One can fix this technicality, but the details are tedious and are often glossed over entirely in the literature (see for example the $B_2$ and $C_2$ examples for $p=2$ \cite{MR3611719}).
\end{rmk}

    \begin{figure}  
    $$\begin{tikzpicture} [scale=0.42]

\path(-5,-2 )  coordinate (A)  ;
\path(A)--++(0:2.8)  coordinate (B)  ;
 \path(A)--++(60:2.8)  coordinate (C)  ;

\path(A)--++(45:28pt) coordinate (A1)  ;
\path(A)--++(-45:28pt) coordinate (A2)  ;
\path(A)--++(-135:28pt) coordinate (A3)  ;
\path(A)--++(135:28pt) coordinate (A4)  ;

\path(B)--++(45:28pt) coordinate (B1)  ;
\path(B)--++(-45:28pt) coordinate (B2)  ;
\path(B)--++(-135:28pt) coordinate (B3)  ;
\path(B)--++(135:28pt) coordinate (B4)  ;

\draw[very thick,rounded corners] (A1)--(B1)--(B2)--(A3)--(A4)--(A1);

\draw[very thick](A)--(B)--(C)--(A);
 \draw[fill=magenta,very thick] (C) circle (12pt); 
  \draw[fill=cyan,very thick] (B) circle (12pt); 
   \draw[fill=green!80!black,very thick] (A) circle (12pt); 
 
\draw[magenta, line width=3] (2,-2 )--(2,0.5) coordinate (hi0)  ;

\draw[cyan, line width=3] (hi0)--++(150:2.5)  coordinate (hi1);
\draw[green!80!black, line width=3] (hi0)--++(30:2.5)  coordinate (hi2);
\draw[green!80!black, line width=3] (hi1)--++(90:2.5)  coordinate (hi3);
\draw[cyan, line width=3] (hi2)--++(90:2.5)  coordinate (hi4);

\draw[magenta, line width=3] (hi3)--++(150:2.5)  coordinate (hi5);
\draw[cyan, line width=3] (hi3)--++(30:2.5)  coordinate (hi6);
\draw[green!80!black , line width=3] (hi4)--++(150:2.5)  coordinate (hi6);
\draw[magenta, line width=3] (hi4)--++(30:2.5)  coordinate (hi7);

\draw[cyan, line width=3] (hi5)--++(90:2.5)  coordinate (hi8);
\draw[magenta, line width=3] (hi6)--++(90:2.5)  coordinate (hi9);

 \draw[green!80!black, line width=3] (hi7)--++(90:2.5)  coordinate (hi10);

\draw[magenta, line width=3] (hi8)--++(30:2.5)  coordinate (hi11);
\draw[cyan, line width=3] (hi9)--++(150:2.5) coordinate (hi14); 

\draw[green!80!black, line width=3, 
 ] (hi8)--++(150:2.5) coordinate(hi13) ; 
\draw[magenta, line width=3, densely dashed
 ] (hi13)--++(90:2.5)  ;  
\draw[green!80!black, line width=3,   
 ] (hi14)--++(90:2.5)coordinate (hi15)   ;  
\draw[cyan, line width=3,   densely dashed
 ] (hi15)--++(30:2.5)   ;  
\draw[magenta, line width=3,   densely dashed
 ] (hi15)--++(150:2.5)   ;

\draw[green!80!black , line width=3] (hi9)--++(30:2.5)  coordinate (hi12);
\draw[magenta, line width=3] (hi10)--++(150:2.5) ;

\draw[cyan, line width=3,   
 ] (hi12)--++(90:2.5)coordinate (hi16)   ;  
\draw[magenta, line width=3,   densely dashed
 ] (hi16)--++(30:2.5)   ;  
\draw[green!80!black, line width=3,   densely dashed
 ] (hi16)--++(150:2.5)   ;

\draw[cyan, line width=3,    
 ] (hi10)--++(30:2.5) coordinate(hi17)  ;  
\draw[magenta, line width=3,   densely dashed
 ] (hi17)--++(90:2.5)   ;

  \foreach \i in {0,1,2,3,4,5,...,17}
{
 \fill (hi\i)circle (7pt);  }
\fill(2,0.5)circle (7pt);
\fill(2,-2)circle (7pt);

\end{tikzpicture}$$  
\caption{\color{black}The Coxeter and weak Bruhat graphs for $\mathfrak{S}_3\leq \widehat{\mathfrak{S}}_3$.  The weak  Bruhat graph can be thought of as a hexagonal-tiling of a sixth of the plane. }
\label{bruhhhhhat}
\end{figure}    

 Fix   $x  \in {^P}W$ and $\x$ a reduced word.     
 We define  right ``cell" ideals 
\begin{align}
\mathcal{H}_{P\backslash W}^{ \leq  \x  }  = {\sf1}_{	\leq \x	} \mathcal{H}_{P\backslash W}^\Bbbk  
 &&
\mathcal{H}_{P\backslash W}^{< \x}  =
    \mathcal{H}_{P\backslash W}^{ \leq  \x }   \cap \ZZ \{ c_{\sts  \stt}  \mid 
 \sts,\stt \in \Std^P
 ( y ), y <   x \} .  
 \end{align}
We    define the   {\sf standard}     $\mathcal{H}_{P\backslash W}^\Bbbk$-module, $\Delta_{\mathcal{H}_{P\backslash W}}(x)$,  
as follows: 
\begin{equation}  \label{identification}
        \Delta_{\mathcal{H}_{P\backslash W}} ( x) =  \mathcal{H}_{P\backslash W}^{ \leq  \x  }  / \mathcal{H}_{P\backslash W}^{< \x}  = \{ c_\sts  + \mathcal{H}_{P\backslash W}^{< \x}   
       \mid \sts \in \Std^P ( x)\} . \end{equation}
We will almost always drop the subscript and simply write
$ \Delta (x) := \Delta_{\mathcal{H}_{P\backslash W}}  (x)$.  
    We recall that the   cellular structure allows us to define, for each
 $ x \in {^PW}$,  a bilinear form 
  $\langle\ ,\ \rangle^{ x} $ 
   on $\Delta(x) $    which
is determined by
\begin{equation}\label{geoide}
   c_{\sts \stt}  c_{\stu \stv}\equiv
  \langle c _\stt,c _\stu \rangle^ x  c_{\sts\stv}
  \pmod{ \mathcal{H}_{P\backslash W}^{< \x}} 
  \end{equation}
for any $\sts,\stt,\stu,\stv \in \Std^P (x  )$.  
When $\Bbbk$ is a field,  we obtain a complete set of non-isomorphic simple
  $\mathcal{H}_{P\backslash W}^\Bbbk$-modules 
  $ L  ( x) $  for $x \in {^PW} $  via exact sequences 
  as follows:
\begin{align}\label{fuckoffre}
 0\to 
  \rad(	\langle\ ,\ \rangle^{ x} 			 ) \rightarrow  \Delta  ( x) \to L  ( x)  \to 0.
\end{align}
 We will not discuss  projective indecomposable $\mathcal{H}_{P\backslash W}^\Bbbk$-modules\footnote{``That whereof we cannot speak, thereof we must remain silent", Wittgenstein.}.
The right  $1_{\leq \w} \mathcal{H}_{P\backslash W}^\Bbbk1_{\leq \w}  $-modules
 \begin{align}
 \Delta_{\leq \w}(x):= \Delta (x){\sf 1}_{\leq \w}\qquad 
L_{\leq \w}(x):= L (x){\sf 1}_{\leq \w}
 \end{align}
for $x \in \mathcal{P}_{\leq \w}$ provide complete sets of non-isomorphic standard and simple modules respectively.  
 The projective indecomposable $1_{\leq \w} \mathcal{H}_{P\backslash W}^\Bbbk1_{\leq \w}  $-modules $P_{\leq w} ( x)$ are the direct summands 
\begin{align}
1_{\leq\w}  \mathcal{H}_{P\backslash W}^\Bbbk1_{\leq \w}  = 
\bigoplus_{ x\leq  w	}
\dim_v( L_{\leq w}(x))  P_{\leq w} (x). 
\end{align}
    
  \color{black}

\begin{assump} 
For the remainder of the paper, we will assume that $\Bbbk$ is a field of characteristic $p\neq 2$.   
\end{assump}
 \color{black}


\begin{eg}\label{heresmyexample}
\color{black}
Let $W$ be the affine symmetric group $\widehat{\mathfrak{S}}_3$ with generators 
$\color{magenta}s_1$, $\color{cyan}s_2$, $\color{green!80!black}s_3 $
and let $P$ be the maximal finite parabolic generated by 
 $\color{cyan}s_2$ and  $\color{green!80!black}s_3$.  The Coxeter graph and Bruhat graphs are depicted in \cref{bruhhhhhat}.  
For $\w= \color{magenta}s_1 \color{cyan}s_2 \color{green!80!black}s_3\color{magenta}s_1 $, the algebra 
 $ {\sf 1}_{ \leq \w }\mathcal{H}_{P\backslash W}^\Bbbk {\sf 1}_{ \leq \w } $ has graded dimension
 $$1^2+(\grade +1)^2+(\grade +1)^2+(\grade +1)^2+(\grade ^2+4\grade +3+\grade ^{-1})^2+(2\grade ^2+3\grade +1)^2 $$
where the sum is over the squares of the graded dimensions of the standard modules, which have labels
$$
\color{magenta}s_1 \color{cyan}s_2 \color{green!80!black}s_3\color{magenta}s_1 
\qquad
\color{magenta}s_1 \color{cyan}s_2 \color{green!80!black}s_3
\qquad
\color{magenta}s_1 \color{cyan}s_2
\qquad
\color{magenta}s_1  \color{green!80!black}s_3
\qquad
\color{magenta}s_1  
\qquad
 \color{black}
1_{P\backslash W}$$respectively.  For example, the basis of the $9$-dimensional standard module has basis  
\begin{align}\begin{split}
\label{9/}
\begin{tikzpicture}[yscale=1,xscale=1]
 \draw[densely dotted,rounded corners](-0.25,0)
  rectangle (1.75,1.4);
   \draw[line width=0.06cm, magenta](0,0) --++(90:1.4);
   \draw[line width=0.06cm, cyan,fill= cyan](0.5,0) 
  --++(90:0.55) circle (2.5pt);    \draw[line width=0.06cm, green!80!black,fill=green!80!black](1,0) --++(90:0.55) circle (2.5pt);
 \end{tikzpicture}
 \qquad\qquad
  \begin{tikzpicture}[yscale=1,xscale=1]
 \draw[densely dotted,rounded corners](-0.25,0)
  rectangle (1.25,1.4);   
   \draw[line width=0.06cm, magenta](0,0) --++(90:1.4);
      \draw[line width=0.06cm, green!80!black,fill=green!80!black](0.5,0) 
  --++(90:0.55) circle (2.5pt);  
 \end{tikzpicture}  
 \qquad\qquad
  \begin{tikzpicture}[yscale=1,xscale=1]
 \draw[densely dotted,rounded corners](-0.25,0)
  rectangle (1.25,1.4);    
   \draw[line width=0.06cm, magenta](0,0) --++(90:1.4);
      \draw[line width=0.06cm, cyan,fill= cyan](0.5,0) 
  --++(90:0.55) circle (2.5pt);
  \end{tikzpicture}
 \qquad\qquad
  \begin{tikzpicture}[yscale=1,xscale=1]
\draw[densely dotted,rounded corners](-0.25,0)
  rectangle (0.75,1.4);
   \draw[line width=0.06cm, magenta](0,0) --++(90:1.4); 
 \end{tikzpicture}  
 \qquad\qquad
  \begin{tikzpicture}[yscale=1,xscale=1]
\draw[densely dotted,rounded corners](-0.25,0)
  rectangle (0.75,1.4); 
   \draw[line width=0.06cm, magenta,fill=magenta](0,0)   --++(90:0.55) circle (2.5pt);
     \draw[line width=0.06cm, magenta](0.5,0) --++(90:1.4);
 \end{tikzpicture}  
\\
  \begin{tikzpicture}[yscale=1,xscale=1]
\draw[densely dotted,rounded corners](-0.25,0)
  rectangle (1.75,1.4);
   \draw[line width=0.06cm, magenta](0,0) --++(90:1.4);
   \draw[line width=0.06cm, cyan,fill= cyan](0.5,0) 
  --++(90:0.55) circle (2.5pt);    \draw[line width=0.06cm, green!80!black,fill=green!80!black](1,0) --++(90:0.55) circle (2.5pt);
   \draw[line width=0.06cm, magenta](1.5,0)--++(90:0.2) to [out=90,in=0](0,1);
 \end{tikzpicture}  
   \qquad\qquad
 \begin{tikzpicture}[yscale=1,xscale=1]
\draw[densely dotted,rounded corners](-0.25,0)
  rectangle (1.25,1.4);
    \draw[line width=0.06cm, magenta](1,0)--++(90:0.2) to [out=90,in=0](0,1); who 
   \draw[line width=0.06cm, magenta](0,0) --++(90:1.4);
      \draw[line width=0.06cm, green!80!black,fill=green!80!black](0.5,0) 
  --++(90:0.55) circle (2.5pt);
  \end{tikzpicture}  \qquad\qquad
 \begin{tikzpicture}[yscale=1,xscale=1]
\draw[densely dotted,rounded corners](-0.25,0)
  rectangle (1.25,1.4);
    \draw[line width=0.06cm, magenta](1,0)--++(90:0.2) to [out=90,in=0](0,1);
   \draw[line width=0.06cm, magenta](0,0) --++(90:1.4);
      \draw[line width=0.06cm, cyan,fill= cyan](0.5,0) 
  --++(90:0.55) circle (2.5pt);
 \end{tikzpicture}
  \qquad\qquad
 \begin{tikzpicture}[yscale=1,xscale=1]
\draw[densely dotted,rounded corners](-0.25,0)
  rectangle (0.75,1.4);
    \draw[line width=0.06cm, magenta](0.5,0) to [out=90,in=0](0,0.8);
   \draw[line width=0.06cm, magenta](0,0) --++(90:1.4);
 \end{tikzpicture}  
 \qquad\qquad
\begin{tikzpicture}[yscale=1,xscale=1]
\path(-0.25,0)
  rectangle (0.75,1.4);  
 \end{tikzpicture}  
 \end{split}
\end{align}

\end{eg}

\color{black}
We have that $\mathcal{H}_{P\backslash W}$ is the quotient of 
$\mathcal{H}_W$ by the 2-sided ideal 
\begin{align}
J= \langle 
{\sf 1}_\csigma \otimes {\sf 1}_\w \mid\csigma \in S_P, 
 \w \in {\sf exp}(w)   \text{ for $w\in W$}\rangle.
 \end{align}
We have a functor
\begin{align}\begin{split}
&\pi :\mathcal{H}_W{\rm -mod} \to  \mathcal{H}_{P\backslash W}{\rm -mod}\\
&\pi :M \mapsto M/ JM.
\end{split}
\end{align}
We have that 
\begin{align}
\pi (\Delta_{\mathcal{H}_W}(x))=
\begin{cases}
 \Delta_{\mathcal{H}_{P\backslash W}}(x) 
&\text{for $x \in {^P}W$} \\
0 & \text{otherwise}\\
\end{cases}
\quad
\pi ( L_{\mathcal{H}_W}(x) )=
\begin{cases}
 L_{\mathcal{H}_{P\backslash W}}(x) 
& \text{for $x \in {^P}W$} \\
0 & \text{otherwise}\\
\end{cases}
\end{align} 
This is simply because every basis element of the standard module  is killed by taking this quotient (this is mentioned explicitly in the proof of \cite[Theorem~5.3]{antiLW}).

\color{black}


\color{black}
\begin{eg}\label{for-the-ref} 
We continue   with \cref{heresmyexample}, \color{black}we note that 
$\color{magenta}s_1 \color{cyan}s_2 \color{black} \in {^P}W$ and that   $\color{magenta}s_1 \color{cyan}s_2\color{magenta}s_1 \color{black}\in P$.  
 For those unfamiliar with   anti-spherical  light leaves, we now illustrate the manner in which a basis element of $ \mathcal{H}_{   W}^\Bbbk$ 
will die when we take the quotient $ \mathcal{H}_{P\backslash W}^\Bbbk$.  
 We have that the  ${\sf 1}_{\leq \color{magenta}s_1 \color{cyan}s_2 \color{green!80!black}s_3\color{magenta}s_1 } \mathcal{H}_{  W}^\Bbbk{\sf 1}_{\leq \color{magenta}s_1 \color{cyan}s_2 \color{green!80!black}s_3\color{magenta}s_1 } $-module 
 $
\Delta_{\leq \color{magenta}s_1 \color{cyan}s_2 \color{green!80!black}s_3\color{magenta}s_1 }(\color{magenta}s_1 \color{cyan}s_2 \color{black} )	  
  $ is $(\grade^2+2\grade+1)$-dimensional and spanned by the diagrams
\begin{equation}\label{ghghghghg}
\begin{minipage}{2cm} \begin{tikzpicture}[yscale=1,xscale=1]
\draw[densely dotted,rounded corners](-0.25,0)
  rectangle (1.75,1.4);
 
   \draw[line width=0.06cm, magenta](0,0) --++(90:1.4);
   \draw[line width=0.06cm, cyan,fill= cyan](0.5,0) 
  --++(90:1.4);
    \draw[line width=0.06cm, green!80!black,fill=green!80!black](1,0) --++(90:0.55) circle (2.5pt);

   \draw[line width=0.06cm, magenta,fill=magenta](1.5,0)--++(90:0.55) circle (2.5pt);  
 \end{tikzpicture}  \end{minipage}
 \qquad
\begin{minipage}{1.5cm} \begin{tikzpicture}[yscale=1,xscale=1]
\draw[densely dotted,rounded corners](-0.25,0)
  rectangle (1.25,1.4);
 
   \draw[line width=0.06cm, magenta](0,0) --++(90:1.4);
   \draw[line width=0.06cm, cyan,fill= cyan](0.5,0) 
  --++(90:1.4);


   \draw[line width=0.06cm, magenta,fill=magenta](1,0)--++(90:0.55) circle (2.5pt);  
 \end{tikzpicture}  \end{minipage}
 \qquad
 \begin{minipage}{1.5cm} \begin{tikzpicture}[yscale=1,xscale=1]
\draw[densely dotted,rounded corners](-0.25,0)
  rectangle (1.25,1.4);
 
   \draw[line width=0.06cm, magenta](0,0) --++(90:1.4);
   \draw[line width=0.06cm, cyan,fill= cyan](0.5,0) 
  --++(90:1.4);

 
    \draw[line width=0.06cm, green!80!black,fill=green!80!black](1,0) --++(90:0.55) circle (2.5pt);
 \end{tikzpicture}  \end{minipage}
 \qquad
 \begin{minipage}{1cm} \begin{tikzpicture}[yscale=1,xscale=1]
\draw[densely dotted,rounded corners](-0.25,0)
  rectangle (0.75,1.4);
 
   \draw[line width=0.06cm, magenta](0,0) --++(90:1.4);
   \draw[line width=0.06cm, cyan,fill= cyan](0.5,0) 
  --++(90:1.4);

 
  \end{tikzpicture}  \end{minipage}\; ,
\end{equation}
 however  the  ${\sf 1}_{\leq \color{magenta}s_1 \color{cyan}s_2 \color{green!80!black}s_3\color{magenta}s_1 }  \mathcal{H}_{P\backslash W}^\Bbbk{\sf 1}_{\leq \color{magenta}s_1 \color{cyan}s_2 \color{green!80!black}s_3\color{magenta}s_1 } $-module 
$ \Delta_{\leq \color{magenta}s_1 \color{cyan}s_2 \color{green!80!black}s_3\color{magenta}s_1 }(\color{magenta}s_1 \color{cyan}s_2 \color{black} )	  	  $
is $( \grade+1)$-dimensional module with basis 
\begin{equation}\label{ghghghghg2} \begin{minipage}{1.5cm} \begin{tikzpicture}[yscale=1,xscale=1]
\draw[densely dotted,rounded corners](-0.25,0)
  rectangle (1.25,1.4);
 
   \draw[line width=0.06cm, magenta](0,0) --++(90:1.4);
   \draw[line width=0.06cm, cyan,fill= cyan](0.5,0) 
  --++(90:1.4);

 
    \draw[line width=0.06cm, green!80!black,fill=green!80!black](1,0) --++(90:0.55) circle (2.5pt);
 \end{tikzpicture}  \end{minipage}
 \qquad
 \begin{minipage}{1cm} \begin{tikzpicture}[yscale=1,xscale=1]
\draw[densely dotted,rounded corners](-0.25,0)
  rectangle (0.75,1.4);
 
   \draw[line width=0.06cm, magenta](0,0) --++(90:1.4);
   \draw[line width=0.06cm, cyan,fill= cyan](0.5,0) 
  --++(90:1.4);

 
  \end{tikzpicture}  \end{minipage}\;.\end{equation}
We illustrate how this works with the second diagram  in (\ref{ghghghghg}).
 We rewrite the second  diagram  using the Jones--Wenzl relation (depicted explicitly in \cref{sahkdgkshjagdjkshfgsdjhk}) as follows
\begin{equation}
\begin{gathered} \begin{tikzpicture}[yscale=1,xscale=1]
\draw[densely dotted,rounded corners](-0.25,0)
  rectangle (1.25,2);
 \clip(-0.25,0)
  rectangle (1.25,2);
   \draw[line width=0.06cm, magenta](0,0) --++(90:2.5);
   \draw[line width=0.06cm, cyan,fill= cyan](0.5,0) 
  --++(90:2.5);


   \draw[line width=0.06cm, magenta,fill=magenta](1.0,0)--++(90:0.55) circle (2.5pt);  
 \end{tikzpicture}  \end{gathered}
\; =\;
 \begin{gathered} \begin{tikzpicture}[yscale=1,xscale=1]
\draw[densely dotted,rounded corners](-0.25,0)
  rectangle (1.25,2);
\clip(-0.25,0)
  rectangle (1.25,2);

   \draw[line width=0.06cm, cyan,fill= cyan](0.5,0) 
  --++(90:1.0);



          \draw[line width=0.06cm, cyan] (0.5,1.0) to [out=150, in=-90] (0,1.5);      
             \draw[line width=0.06cm, cyan] (0.5,1.0) to [out=30, in=-90] (1,1.5);  

      \draw[line width=0.06cm, magenta ] (0.5,1.0) to [out=-30, in=90] (1,0.5)--(1.0,0);       
      \draw[line width=0.06cm, magenta ] (0.5,1.0) to [out=-150, in=90] (0,0.5)--(0,00); 
       \draw[line width=0.06cm, magenta ] (0.5,1.0) to [out=90, in=-90] (0.5,2);

%
%

          \draw[line width=0.06cm, cyan](1,1.5)--++(90:3) circle (2.5pt);  
          \draw[line width=0.06cm, magenta,fill=magenta](0.5,2)--++(90:2.5);        
            \draw[line width=0.06cm, cyan,fill=cyan] (0,1.5) circle (2.5pt); 
       
  \end{tikzpicture}  \end{gathered}
\;-\;
 \begin{gathered} \begin{tikzpicture}[yscale=1,xscale=1]
\draw[densely dotted,rounded corners](-0.25,0)
  rectangle (1.25,2);\clip(-0.25,0)
  rectangle (1.25,2);
 
   \draw[line width=0.06cm, magenta](0,0) --++(90:2);
   \draw[line width=0.06cm, cyan,fill= cyan](0.5,0) 
  --++(90:0.5 ) circle (2.5pt);;
   \draw[line width=0.06cm, cyan,fill= cyan](0.5,2) 
  --++(-90:0.5 ) circle (2.5pt);;
 


      \draw[line width=0.06cm, magenta ](1.0 ,0) to[out=90,in=0] (0,1.0);

%
%

%
%
%

          \draw[line width=0.06cm, magenta,fill=magenta](0,2)--++(90:0.5);          \draw[line width=0.06cm, cyan] (0.5,2)--++(90:0.5);
       
  \end{tikzpicture}  \end{gathered}\;.
\end{equation}
We recall that 
$\color{cyan}s_2 \color{black} \in P$; therefore  
the first diagram is now zero by the ${\sf 1}_ {\color{cyan}s_2  }		\otimes D=0$ relation and the second diagram 
belongs to the 2-sided cell-ideal generated by ${\sf 1}_{\color{magenta}s_1}$.  

 In fact, \emph{no} diagram   in (\ref{ghghghghg2}) has a rightmost $\color{magenta}s_1 \color{black}$-strand.
 In more detail, we have ${\color{magenta} s_1}{\color{cyan} s_2}<{\color{magenta} s_1}{\color{cyan} s_2}{\color{magenta} s_1} \notin {^PW}$, so  
 for any tableau $\stt$ of shape ${\color{magenta} s_1}{\color{cyan} s_2}$ and weight ending in ${\color{magenta} s_1}$, the light leaf $c_{\stt}$ factors through ${\sf 1}_{{\color{magenta} s_1}{\color{cyan} s_2}} \otimes {\sf spot}_{\color{magenta} s_1}$.
 By the above reasoning $c_{\stt}$ must therefore vanish in $\mathcal{H}_{P \backslash W}$.
 More generally, if $y<y\ctau \notin {^PW}$ then there are no parabolic tableaux of shape $y$ and weight ending in $\ctau$.
    This    illustrates an important aspect of the restriction functors we will define in the next section, and we will refer back to this example in the proof of \cref{plazaryom}.  
\end{eg}

 \color{black}

\subsection{Branching rules for standard modules}
\color{black}
We define the $\ctau$-restriction functor 
\begin{align}\begin{split}
\Res_{\leq \w}^{\leq \w\ctau} &: {\sf 1}_{\leq \w\ctau} \mathcal{H}_{P\backslash W}^\Bbbk{\sf 1}_{\leq \w\ctau}   {\rm -mod} \xrightarrow{ \ \ \ } {\sf 1}_{\leq \w } \mathcal{H}_{P\backslash W}^\Bbbk{\sf 1}_{\leq \w }{\rm -mod}\\
&:M \mapsto	\res( M ({\sf 1}_{\leq \w}\otimes {\sf 1}_\ctau)   )
\end{split}
\intertext{where }
\res &:
{\sf 1}_{\leq \w\ctau}  
 \mathcal{H}_{P\backslash W}^\Bbbk
{\sf 1}_{\leq \w\ctau}   {\rm -mod}
  \longrightarrow {\sf 1}_{\leq \w } \mathcal{H}_{P\backslash W}^\Bbbk{\sf 1}_{\leq \w }{\rm -mod}
\end{align}
denotes the ordinary restriction functor defined by the embedding of \cref{preimage}.
 While we use the term restriction 
 for $\Res_{\leq \w}^{\leq \w\ctau} $, we emphasise that 
 this functor does not preserve the underlying vector space. 
This is because the inclusion \ref{preimage} is the non-unital inclusion of an idempotent subalgebra, so restriction kills the complementary idempotent.

 \Cref{LEW1} has the following immediate corollary.   
 
\begin{cor} 
\label{plazaryom}
\color{black}
Let     $x=  y\ctau >y$, with $x, y  \in {^P}W$.       
 We have that 
\begin{align}\label{one11}
   0\to  \Bbbk\{c_{\stt^+}\mid \stt \in \Std_{\leq \w}(y)\} \to 
 \Res^{ {\leq \w \ctau}  }_{  {\leq \w  } }(   \Delta_{\leq \w\ctau}(x))
     \to 
     \Bbbk \{c_{\stt^+}\mid \stt \in \Std_{\leq \w}(x)\} \to 0
\intertext{
   and 
}\label{two11}
   0\to \Bbbk\{c_{\stt^-}\mid \stt \in \Std_{\leq \w}(y)\} \to 
\Res^{ {\leq \w \ctau}  }_{  {\leq \w  } }(    \Delta_{\leq \w\ctau}(y))     \to 
      \Bbbk\{c_{\stt^-}\mid \stt \in \Std_{\leq \w}(x)\} \to 0
 \end{align}
  where  
  the submodule is isomorphic to $\Delta_{\leq \w}(y)$ (respectively $\Delta_{\leq \w}(y)\langle 1 \rangle$) and the quotient module is    
    isomorphic to $\Delta_{\leq \w}(x)\langle -1 \rangle$ (respectively $\Delta_{\leq \w}(x)$).  
 Finally, if    $  y< y\ctau=x   \not\in  {^PW}$, then
\begin{align}\label{three11}
 \Res_{\leq \w}^{\leq \w\ctau}(\Delta_{\leq \w\ctau}(y))
 =
0.
\end{align}

     \end{cor}
     
     
\begin{proof}
 We have that the maps  
\begin{align}
   \Delta_{\leq \w }(y)\hookrightarrow  \Res^{ {\leq \w \ctau}  }_{  {\leq \w  } }(   \Delta_{\leq \w\ctau}(x))
   \quad&
 \phi  : c_\stt \mapsto c_{\stt}\otimes {\sf 1}_\ctau
\\
   \Delta_{\leq \w }(y)\hookrightarrow  \Res^{ {\leq \w \ctau}  }_{  {\leq \w  } }(   \Delta_{\leq \w\ctau}(y))
   \quad& \phi: c_\stt \mapsto c_{\stt }\otimes {\sf spot}^\emptyset_\ctau
\end{align}
for $\stt \in \Std_{\leq \w}(y)$ are    injective  $({\sf 1}_{\leq \w }\otimes {\sf1}_\ctau) \mathcal{H}_{P\backslash W}^\Bbbk({\sf 1}_{\leq \w }\otimes {\sf1}_\ctau) $-homomorphisms by construction.   Similarly, 
the maps  
\begin{align}
 \Delta_{\leq \w }(x) \to 
   \Res^{ {\leq \w \ctau}  }_{  {\leq \w  } }(   \Delta_{\leq \w\ctau}(x))/ 
    \phi(   \Delta_{\leq \w }(y))
   \quad&
   : c_\stt \mapsto (c_{\stt}\otimes {\sf 1}_\ctau)({\sf 1}_\y  \otimes {\sf fork}^{\ctau\ctau}_{\ctau})
\\
   \Delta_{\leq \w }(x) \to  
      \Res^{ {\leq \w \ctau}  }_{  {\leq \w  } }(   \Delta_{\leq \w\ctau}(y))/
          \phi(   \Delta_{\leq \w }(y))
   \quad& : c_\stt \mapsto (c_{\stt}\otimes {\sf 1}_\ctau)({\sf 1}_\y  \otimes ({\sf fork}^{\ctau\ctau}_{\ctau} {\sf spot}^\ctau_\emptyset )
\end{align}
for $\stt \in \Std_{\leq \w}(x)$ are  
   $({\sf 1}_{\leq \w }\otimes {\sf1}_\ctau) \mathcal{H}_{P\backslash W}^\Bbbk({\sf 1}_{\leq \w }\otimes {\sf1}_\ctau) $-isomorphisms by construction (as we have simply  multiplied on the right by an element of the algebra).
   \color{black} 
\color{black} On the other hand, if $y<y\ctau=x \notin {^PW}$, then $\Delta_{\leq \w\ctau}(y)({\sf 1}_{\leq \w} \otimes {\sf 1}_{\ctau})=0$ by the light leaves construction, so $\Res_{\leq \w}^{\leq \w\ctau}(\Delta_{\leq \w\ctau}(y))=0$.  We refer  back to \cref{for-the-ref} for an illustrative example.
 \end{proof}

 \begin{eg}\label{heresmyexample2}\color{black}
 We continue with the notation of \cref{heresmyexample}.
 We consider the restriction functor $\Res^{\leq \x \csigma}_{\leq \x}$
  for 
  $\csigma={\color{magenta}s_1}\in W$ and $\x= \color{magenta}s_1 \color{cyan}s_2 \color{green!80!black}s_3$ 
  (note that $\x\csigma=\w$ for $\w$ as in   \cref{heresmyexample}).  
The southern reading words of the first three diagrams in \cref{9/} are 
not of the form $\y\csigma$ for some $y\leq x$ and therefore these diagrams are sent to zero by the restriction functor.    
The remaining  diagrams in the first row of \cref{9/} form a submodule, isomorphic to 
$\Delta_{\leq \color{magenta}s_1 \color{cyan}s_2 \color{green!80!black}s_3}(1_{P\backslash W})$, the isomorphism is depicted on basis elements in \cref{inj}.  
 The  diagrams from the second row of \cref{9/} form a quotient module, isomorphic to 
$\Delta_{\leq \color{magenta}s_1 \color{cyan}s_2 \color{green!80!black}s_3}(		\color{magenta}s_1\color{black} 	)$, the isomorphism is depicted on basis elements in \cref{inj2}.

 \color{black}
 The reader will notice that this sequence is non-split. 
To see this, note that the monochrome diagrams from the first row of \cref{9/} are both obtained from the monochrome diagram in the second row as follows:
$$\color{black}	{\sf 1}_\csigma= {\sf fork}_{\csigma\csigma}^\csigma ({\sf spot}^\csigma_\emptyset\otimes {\sf 1}_\csigma)
\qquad\text{and}\qquad
({\sf spot}_\csigma^\emptyset \otimes {\sf 1}_\csigma)= {\sf fork}_{\csigma\csigma}^\csigma ({\sf spot}^\csigma_\emptyset
{\sf spot}_\csigma^\emptyset\otimes {\sf 1}_\csigma).$$
Whereas, there does not exist any   $D \in {\sf 1}_{\leq \color{magenta}s_1 \color{cyan}s_2 \color{green!80!black}s_3} \mathcal{H}^\Bbbk_{P\backslash W}{\sf 1}_{\leq \color{magenta}s_1 \color{cyan}s_2 \color{green!80!black}s_3}$ such that either 
$$\color{black} {\sf 1}_\csigma (D\otimes {\sf 1}_\csigma) = {\sf fork}_{\csigma\csigma}^\csigma
\qquad\text{or}\qquad
 ({\sf spot}^\csigma_\emptyset\otimes {\sf 1}_\csigma) (D\otimes {\sf 1}_\csigma) = {\sf fork}_{\csigma\csigma}^\csigma.
$$
 \color{black}

 \begin{figure}[ht!]\color{black}
$$
 \Bigg\{ \; \begin{minipage}{1.1cm} \begin{tikzpicture}[yscale=1,xscale=1]
\draw[densely dotted,rounded corners](-0.25,0)
  rectangle (0.75,1.4);
 
  
 \end{tikzpicture}  \end{minipage}
\;  , \; 
   \begin{minipage}{1.1cm}
  \begin{tikzpicture}[yscale=1,xscale=1]
\draw[densely dotted,rounded corners](-0.25,0)
  rectangle (0.75,1.4);
 
   \draw[line width=0.06cm, magenta,fill=magenta](0,0)   --++(90:0.55) circle (2.5pt);
 \end{tikzpicture} \end{minipage}\; \Bigg\}
 %
 %
\;
  \begin{minipage}{1.5cm} \begin{tikzpicture}[yscale=1,xscale=1]
\path (-0.5,0)
  rectangle (1,1.4) ;
  \path  (0.25,0.7)coordinate  (hi);

 \draw(hi) node [above] {$\otimes {1}_\csigma$} ;
 \draw(hi) node {$\xrightarrow{ \ \ \ \ \ \ \ \ } $};
  
 \end{tikzpicture} \end{minipage}\; %
 \Bigg\{ \; \begin{minipage}{1.1cm} \begin{tikzpicture}[yscale=1,xscale=1]
\draw[densely dotted,rounded corners](-0.25,0)
  rectangle (0.75,1.4);
 
   \draw[line width=0.06cm, magenta](0.5,0) --++(90:1.4);
  
 \end{tikzpicture}  \end{minipage}
\;  , \; 
   \begin{minipage}{1.1cm}
  \begin{tikzpicture}[yscale=1,xscale=1]
\draw[densely dotted,rounded corners](-0.25,0)
  rectangle (0.75,1.4);
 
   \draw[line width=0.06cm, magenta,fill=magenta](0,0)   --++(90:0.55) circle (2.5pt);
     \draw[line width=0.06cm, magenta](0.5,0) --++(90:1.4);
 \end{tikzpicture} \end{minipage}\; \Bigg\}$$

 \caption{\color{black}The isomorphism
 $\Delta_{\leq \color{magenta}s_1 \color{cyan}s_2 \color{green!80!black}s_3}(1_{P\backslash W})\cong 
\{c_{\stt^+}\mid \stt \in \Std_{\leq \color{magenta}s_1 \color{cyan}s_2 \color{green!80!black}s_3 }(1_{P\backslash W}  )\}  $  as in \cref{heresmyexample2}.  
The righthand-side   forms the submodule of $
\Res^{{\leq \color{magenta}s_1 \color{cyan}s_2 \color{green!80!black}s_3 \color{magenta}s_1}}_{{\leq  \color{magenta}s_1 \color{cyan}s_2 \color{green!80!black}s_3 }}(
 \Delta_{\leq \color{magenta}s_1 \color{cyan}s_2 \color{green!80!black}s_3 \color{magenta}s_1}(  \color{magenta}s_1 \color{black})) $
  in the short exact sequence of \ref{one11}.  }

\label{inj}
 \end{figure}

 \begin{figure}[ht!]\color{black}

  $$  
   \Bigg\{\;  \begin{minipage}{1.5cm} \begin{tikzpicture}[yscale=1,xscale=1]
\draw[densely dotted,rounded corners](-0.25,0)
  rectangle (1.25,1.4);
 
   \draw[line width=0.06cm, magenta](0,0) --++(90:1.4);
   \draw[line width=0.06cm, cyan,fill= cyan](0.5,0) 
  --++(90:0.55) circle (2.5pt);     
   \draw[line width=0.06cm, green!80!black,fill=green!80!black](1,0) 
  --++(90:0.55) circle (2.5pt);     

 \end{tikzpicture} \end{minipage}  
   \; , \;
  \begin{minipage}{1cm} \begin{tikzpicture}[yscale=1,xscale=1]
\draw[densely dotted,rounded corners](-0.25,0)
  rectangle (0.75,1.4);

   \draw[line width=0.06cm, magenta](0,0) --++(90:1.4);
      \draw[line width=0.06cm, green!80!black,fill=green!80!black](0.5,0) 
  --++(90:0.55) circle (2.5pt);
 
 \end{tikzpicture} \end{minipage}  \; , \;
  \begin{minipage}{1cm} \begin{tikzpicture}[yscale=1,xscale=1]
\draw[densely dotted,rounded corners](-0.25,0)
  rectangle (0.75,1.4);

   \draw[line width=0.06cm, magenta](0,0) --++(90:1.4);
      \draw[line width=0.06cm, cyan,fill= cyan](0.5,0) 
  --++(90:0.55) circle (2.5pt);
 
 \end{tikzpicture} \end{minipage}
  \; , \;
 \begin{minipage}{0.5cm} \begin{tikzpicture}[yscale=1,xscale=1]
\draw[densely dotted,rounded corners](-0.25,0)
  rectangle (0.25,1.4);
 
   \draw[line width=0.06cm, magenta](0,0) --++(90:1.4);
  
 \end{tikzpicture} \end{minipage} \; \Bigg\}
 \begin{minipage}{1.3cm} \begin{tikzpicture}[yscale=1,xscale=1]
\path (-0.5,0)
  rectangle (1,1.4) ;
  \path  (0.25,0.7)coordinate  (hi);

 \draw(hi) node {$\xrightarrow{ \ \ \ \ \ \ } $};
  
 \end{tikzpicture} \end{minipage}\; %
  \Bigg\{\; \begin{minipage}{2cm} \begin{tikzpicture}[yscale=1,xscale=1]
\draw[densely dotted,rounded corners](-0.25,0)
  rectangle (1.75,1.4);
 
   \draw[line width=0.06cm, magenta](0,0) --++(90:1.4);
   \draw[line width=0.06cm, cyan,fill= cyan](0.5,0) 
  --++(90:0.55) circle (2.5pt);    \draw[line width=0.06cm, green!80!black,fill=green!80!black](1,0) --++(90:0.55) circle (2.5pt);

   \draw[line width=0.06cm, magenta](1.5,0)--++(90:0.2) to [out=90,in=0](0,1);
 \end{tikzpicture} \end{minipage}  
   \; , \;
 \begin{minipage}{1.5cm} \begin{tikzpicture}[yscale=1,xscale=1]
\draw[densely dotted,rounded corners](-0.25,0)
  rectangle (1.25,1.4);
    \draw[line width=0.06cm, magenta](1,0)--++(90:0.2) to [out=90,in=0](0,1);

   \draw[line width=0.06cm, magenta](0,0) --++(90:1.4);
      \draw[line width=0.06cm, green!80!black,fill=green!80!black](0.5,0) 
  --++(90:0.55) circle (2.5pt);
 
 \end{tikzpicture} \end{minipage}  \; , \;
 \begin{minipage}{1.5cm} \begin{tikzpicture}[yscale=1,xscale=1]
\draw[densely dotted,rounded corners](-0.25,0)
  rectangle (1.25,1.4);
    \draw[line width=0.06cm, magenta](1,0)--++(90:0.2) to [out=90,in=0](0,1);
 
   \draw[line width=0.06cm, magenta](0,0) --++(90:1.4);
      \draw[line width=0.06cm, cyan,fill= cyan](0.5,0) 
  --++(90:0.55) circle (2.5pt);
 
 \end{tikzpicture} \end{minipage}
  \; , \;
 \begin{minipage}{1cm} \begin{tikzpicture}[yscale=1,xscale=1]
\draw[densely dotted,rounded corners](-0.25,0)
  rectangle (0.75,1.4);
    \draw[line width=0.06cm, magenta](0.5,0) to [out=90,in=0](0,0.8);

   \draw[line width=0.06cm, magenta](0,0) --++(90:1.4);
  
 \end{tikzpicture} \end{minipage} \; \Bigg\}
$$

 \caption{\color{black}The isomorphism
 $\Delta_{\leq \color{magenta}s_1 \color{cyan}s_2 \color{green!80!black}s_3}({\color{magenta}s_1})\cong 
\{c_{\stt^+}\mid \stt \in \Std_{\leq \color{magenta}s_1 \color{cyan}s_2 \color{green!80!black}s_3 }(\color{magenta}s_1 )\}  $  as in \cref{heresmyexample2}.  
The righthand-side   forms the quotient module of  $
\Res^{{\leq \color{magenta}s_1 \color{cyan}s_2 \color{green!80!black}s_3 \color{magenta}s_1}}_{{\leq  \color{magenta}s_1 \color{cyan}s_2 \color{green!80!black}s_3 }}(
 \Delta_{\leq \color{magenta}s_1 \color{cyan}s_2 \color{green!80!black}s_3 \color{magenta}s_1}(  \color{magenta}s_1 \color{black})) $
  in the short exact sequence of \ref{one11}.  }

\label{inj2}
 \end{figure}

 \end{eg}

%
%
%

    \subsection{$p$-Kazhdan--Lusztig polynomials}

 The categorical (rather than geometric) definition of   the  $p$-Kazhdan--Lusztig polynomials is given via the {\em diagrammatic character} of 
 \cite[Definition 6.23]{MR3555156}.   In the language of this paper,  the definition of the   {\sf anti-spherical $p$-Kazhdan--Lusztig polynomial}, ${^p}{n}_{x,y}(v)$  for $x,y\in {^PW}$,   is as follows, 
\begin{align}
{^p}{n}_{x,y}(\grade)
:=
   \sum_{k\in\mathbb Z}
 \dim_v(\Hom_{{\sf1}_{\leq \w} \mathcal{H}_{P\backslash W}^\Bbbk{\sf 1}_{\leq \w}}(P_{\leq w}(x),\Delta_{\leq \w}(y))
 =
\sum_{k\in \mathbb Z}   [\Delta_{\leq \w} (\underline{y})	: L _{\leq \w}(\underline{x})\langle k \rangle ] \grade^k   
 \end{align}
  for any  $x,y\leq w$  and $\underline{x}\in {\sf rex}_P(x),\underline{y}\in {\sf rex}_P(y)$,  $\w\in {\sf rex}_P(w)$ are arbitrary   
   (note that the definition of these polynomials is independent of $\w$).  
The anti-spherical  $p$-Kazhdan--Lusztig polynomials are recorded in the $|{^PW}|\times |{^PW}|$-matrix 
\begin{align}
{\bf D}_{P\backslash W}  =\big( {^pn}_{x,y}(\grade)\big)_{x,y \in {^PW}}    
\end{align}
and we set 
\begin{align}
{\bf D}^{-1} _{P\backslash W} =\big( {^pn}_{x,y}^{-1}(\grade)\big)_{x,y \in {^PW}}    
\end{align}
to be the inverse of this matrix (which exists, as ${\bf D}_{P\backslash W}$ is lower uni-triangular).  The non-parabolic ($p$-)Kazhdan--Lusztig polynomials are obtained by setting $P=\{1_W\}\leq W$.

  \section{The classification and construction of \\ homogeneous $\mathcal{H}^\Bbbk_{P \backslash W}$-modules}
 It is, in general, a hopeless task to attempt to understand all $p$-Kazhdan--Lusztig polynomials or to understand all simple $\mathcal{H}^\Bbbk_{P \backslash W}$-modules.  
In particular, it was shown in \cite{w13} that one can embed certain number-theoretic questions (for which no combinatorial solution could possibly be hoped to exist) into the $p$-Kazhdan--Lusztig matrices of affine symmetric groups.

Thus we restrict our attention to classes of modules which we can hope to understand.  
Over the complex numbers, the first port of  call would be   to attempt to understand the unitary modules; for Lie groups this ongoing project is Vogan's   famous Atlas of Lie groups.  
Over arbitrary fields, the  notion of unitary no longer makes sense; however,  
 for graded algebras  the homogeneous representations  seem to provide a suitable replacement.  
 For quiver Hecke algebras, the homogeneous representations were classified and constructed by Kleshchev--Ram \cite{JEMS}.  For (quiver) Hecke algebras of symmetric groups, the notions of unitary and homogeneous  representations coincide over the complex field \cite[Theorem 8.1]{BNS} and  the  beautiful cohomological and structural 
properties of these (homogeneous) representations are entirely independent of the field  \cite{BNS,JEMS}.

 In this section, we  fix  $W$ an arbitrary Coxeter group and fix $P$ an arbitrary parabolic subgroup and 
 we classify and construct the homogeneous representations of 
 the  diagrammatic Hecke category $\mathcal{H}_{P \backslash W}$.  
 We first provide a cohomological construction of the module $L(1_{P \backslash W})$  via a BGG resolution.  
This cohomological construction allows us to immediately deduce a basis-theoretic construction of $L(1_{P \backslash W})$, from which we easily read-off the fact that $L(1_{P \backslash W})$ is homogeneous.  We then prove that $L(w)$ is inhomogeneous for any $1\neq w \in {^PW}$.

\begin{defn}
Given $w, y\in {^PW}$, we say that   
 $( w, y)$ is a {\sf Carter--Payne pair} if $y\leq w$ and $\ell(y)=\ell(w)-1$.  
 We let ${\sf CP}_\ell$ denote the set of Carter--Payne pairs $(w,y)$ 
 with $\ell(w)=\ell\in \NN$.  
  \end{defn}

For $P\subseteq W$  an affine  Weyl group an its maximal finite parabolic subgroup,  the following family of homomorphisms were first considered (in the context of algebraic groups) by Carter--Payne in \cite{CP}.

\begin{thm}
For   $(w,y) \in {\sf CP}_\ell$, pick an arbitrary reduced expression  $\w=\sigma_1\dots \sigma_\ell$ and 
suppose that  $\y=\sigma_1\dots \sigma_{p-1}\widehat{\sigma}_p \sigma_{p+1}\dots \sigma_\ell$ is the subexpression of $y$ obtained by deleting precisely one element $\sigma_p \in S$.    We have that 
$$ 
\Hom_{\mathcal{H}_{P\backslash W}^\Bbbk}(\pDelta(  w ),\pDelta( y ))
$$ 
is $\grade^1$-dimensional.
Given choices of $\w$ and $\y$ as above
this homomorphism space is
 spanned by the map 
 $$
\varphi^ w_{ y} (c_\stt ) = 
 ({\sf1}_ {\sigma_1\cdots \sigma_{p-1}} \otimes 
  {\sf spot}_{ {\sigma}_p} ^\emptyset 
    \otimes {\sf 1}_{ \sigma_{p+1}\cdots \sigma_\ell}   )
 c_\stt 
$$
for $\stt \in \Std(w)$.  
\end{thm}   

\begin{proof} 
Since $ \mathcal{P}_{\leq  w} $ is a co-saturated subset of ${^P}W$,
 we can truncate to the algebra 
${\sf1}_{\leq \w}\mathcal{H}^\Bbbk_{P \backslash W}{\sf1}_{\leq \w}$ by  \cite[Appendix A3.13]{Donkin}. 
\color{black}
 We have that 
$$
  \pDelta(y) {\sf 1}_{\w} = {\rm Span}_\ZZ \{{\sf1}_ {\sigma_1\cdots \sigma_{p-1}} \otimes 
 {\sf spot}_{ {\sigma}_p} ^\emptyset 
   \otimes {\sf 1}_{ \sigma_{p+1}\cdots \sigma_\ell}  \}
$$
by \cref{LEW1}, 
\color{black}
as there is a unique tableau $\stt$ with shape $y$ \cite[Theorem~5.8]{hum}.  
\color{black}
Moreover this space is of strictly positive degree, namely $\grade^1$.  
Whereas, the character of the simple head, $L(y)$ of $\pDelta(y) $, 
is invariant under swapping $\grade$ and $\grade^{-1}$ by \cite[Proposition 2.18]{hm10}.
  Therefore 
\begin{equation}\label{dfkjldglkfjgd}
    \pDelta  (y){\sf 1}_{\w} =  \rad (  \pDelta  (y){\sf 1}_{\w} ) \quad \text{ and } \quad 
 L (y){\sf 1}_{\w}  =0.\end{equation} 
By our assumption that $(w,y)$ is a Carter--Payne pair, 
 there does not exist an $x\in W$ such that 
$y<  x <  w$. 
We now apply this assumption twice.  Firstly, we note that 
  $  [ \Delta_{\leq \w} (y): L_{\leq \w}(x)]\neq 0$ implies that $y\leq   x \leq   w$.  
Putting this together with \cref{dfkjldglkfjgd} we have that  
$ \rad ( \pDelta_{\leq \w} (y))= L_{\leq \w}( w)\langle 1 \rangle$
and the graded decomposition number is equal to 
$$
 \dim_\grade(\Hom_{{\sf 1}_{\leq \w}\mathcal{H}^\Bbbk_{P \backslash W}{\sf 1}_{\leq \w}}(P_{\leq w}(w), \pDelta_{\leq \w}(y)))= 
 \sum_{k\in \mathbb Z} [\pDelta_{\leq \w} (y) : L_{\leq \w} (w)\langle k \rangle ] = \grade^1 .
$$
Now applying our assumption again,  we conclude that   this homomorphism factors through the projection $P_{\leq w}(w)\to \Delta_{\leq \w}(w)$  by highest weight theory and so we have  
$$
  \dim_\grade(
  \Hom_{{\sf 1}_{\leq \w} \mathcal{H}^\Bbbk_{P \backslash W}{\sf 1}_{\leq \w}}
  (P_{\leq w}(w), \pDelta_{\leq \w}(y)))= 
\dim_\grade(\Hom_{{\sf 1}_{\leq \w}\mathcal{H}^\Bbbk_{P \backslash W}{\sf 1}_{\leq \w}}(\pDelta_{\leq \w}(w), \pDelta_{\leq \w}(y))),
$$ 
and thus the result follows.  
\end{proof}

We set $\mathcal{P}_{ \ell} =  \{ w \in  {^PW}  \mid \ell(w)=\ell \}$ for each $\ell\in \NN$.  
 Following a construction going back to
work of Bernstein--Gelfand--Gelfand 
and Lepowsky \cite{bgg,MR0414645}, we are going to define a complex
of graded $\mathcal{H}_{P\backslash W}^\Bbbk$-modules
\begin{equation}\label{bggres}
\cdots
\longrightarrow
\Delta_2 \stackrel{\delta_2}{\longrightarrow}
\Delta_1 \stackrel{\delta_1}{\longrightarrow}
\Delta_0 \stackrel{\delta_0}{\longrightarrow}
  0,
\end{equation}
where
\begin{equation}\label{modi}
\Delta_\ell := \bigoplus_{
\begin{subarray}c
w \in \mathcal{P}_{\ell} \end{subarray}} 
\Delta(w)\langle \ell(w) \rangle.
\end{equation}
We will refer to this as the {\sf BGG complex}. 
\color{black}
%
We momentarily assume that $P=1\leq W$ is the trivial parabolic (so that ${^PW}=W$).    
\color{black}
 Suppose   $w,  z\in W$ such that
  $\ell(w)=\ell(z)+2$ and 
   $ w>   z$ in the Bruhat ordering.  By \cite[Lemma~10.3]{bgg} there exists a unique pair of distinct elements  $x,y \in W$ such that $w>x,y>z$.  
We refer to the quadruple $w,x,y,z \in W$ as a {\sf diamond} and  we have homomorphisms  of $\mathcal{H}^\Bbbk_W$-modules
\begin{align}\begin{minipage}{5.5cm}
\begin{tikzpicture}
 \path(0,0)--++(90:30pt) coordinate(N);
 \path(0,0)--++(-90:30pt) coordinate(S);
 \path(0,0)--++(0:70pt) coordinate(E);
  \path(0,0)--++(180:70pt) coordinate(W);
\draw(N) node {$\Delta(x)$};
\draw(S) node {$\Delta(y)$};
\draw(W) node {$\Delta(w)$};
\draw(E) node {$\Delta(z)$};
\draw[->](173:70pt)to node[midway,above]  {$\varphi^ w_{ x}$} (120:30pt)  ; 
\draw[->](-173:70pt)to node[midway,below]  {$\varphi^w_{y}$} (-120:30pt)  ; 
\draw[<-](7:70pt)to node[midway,above]  {$\varphi^x_{z}$} (60:30pt)  ; 
\draw[<-](-7:70pt)to node[midway,below]  {$\varphi^y_{z}$} (-60:30pt)  ;   \end{tikzpicture}
\end{minipage}\end{align}
given by our Carter--Payne homomorphisms.  
By a generalization of 
\cite[Lemma 10.4]{bgg} to arbitrary Coxeter groups (see \cite[\S 6.17]{hum}), it is possible to pick
a sign $\eps(\alpha,\beta)$ for each Carter--Payne pair  $(\alpha,\beta)$
such that for every diamond
the product of the signs associated to its four arrows
is equal to $-1$.
%
%
      We can now define the $\mathcal{H}^\Bbbk_{  W}$-differential
$\delta_\ell:\Delta_{\ell } \rightarrow \Delta_{\ell-1}$ for $\ell\geq 1$ 
to be the sum of the maps
\begin{align}
\eps(\alpha,\beta) \varphi^{{\alpha}}_{{\beta}}:\Delta({ {\alpha}})\langle \ell \rangle
\rightarrow \Delta({ {\beta}})\langle  \ell-1  \rangle
\end{align}
over all  Carter--Payne pairs     $(\alpha,\beta)\in {\sf CP}_\ell 
 $.  
We let   $ C_\bullet(1_{  W}) = 
  \bigoplus_{\ell\geq 0 }{  \Delta}_\ell \langle \ell \rangle 
 $ 
together with the differential $(\delta_\ell)_{\ell\geq 0}$.


\color{black}
\begin{lem}\label{lemmmmmmer}
We have that $\varphi^{ w}_{x} \varphi^{x}_ {z}= 
\varphi^{w}_{y} \varphi_{z}^{y} 
$.
\end{lem}
\begin{proof}
We  truncate to consider a subalgebra ${\sf 1}_{\leq \w }  \mathcal{H}_{  W}^\Bbbk {\sf 1}_{\leq \w}$ for $\w=\sigma_1 \sigma_2\cdots \sigma_\ell$ a fixed reduced expression of $w\in W$.  
Since $\ell(x)=\ell(y)=\ell-1$ there are unique subexpressions $\x,\y$ for $x,y$ respectively inside $\w$. 
Similarly there are unique subexpressions for $z$ inside $\x,\y$, which induce subexpressions $\vvv_x,\vvv_y$ for $z$ inside $\w$.
On the other hand, there is a unique diagram in $ \Delta_{\leq \w}(z)1_{\w}$ of maximal degree equal to $\ell(w)-\ell(z)=2$, this diagram is equal to
$$
({\sf1}_ {\sigma_1\cdots \sigma_{t-1}} \otimes 
 {\sf spot}_{ {\sigma}_{t}} ^\emptyset 
   \otimes {\sf 1}_{ \sigma_{t+1} \cdots \sigma_{u-1}} \otimes {\sf spot}_{\sigma_u}^\emptyset \otimes {\sf 1}_{\sigma_{u+1} \dots   \sigma_{\ell}}   )
$$
for some $1\leq t<u\leq \ell$.  
The corresponding reduced expression 
$$
\sigma_1\dots
\sigma_{t-1}\widehat{\sigma}_t \sigma_{t+1}\dots
 \sigma_{u-1}\widehat{\sigma}_u \sigma_{u+1}\dots \sigma_\ell$$
   is equal to (at least) one of the expressions 
 $\vvv_x $ or $ \vvv_y$. 
Without loss of generality, we suppose it is $\vvv_y$; this implies   that 
$$\varphi^{ w}_{y} \varphi_{z}^{y} (c_\stt)= ({\sf1}_ {\sigma_1\cdots \sigma_{t-1}} \otimes 
 {\sf spot}_{ {\sigma}_{t}} ^\emptyset 
   \otimes {\sf 1}_{ \sigma_{t+1} \cdots \sigma_{u-1}} \otimes {\sf spot}_{\sigma_u}^\emptyset \otimes {\sf 1}_{\sigma_{u+1} \dots   \sigma_{\ell}}   ) c_\stt.$$   
   We now consider the other composition $\varphi^{w}_{x} \varphi^{x}_ {z}$.   
Let $1\leq p<q \leq \ell$ denote the indices of the terms $\sigma_p$ and $\sigma_q$ which are deleted from $\w$ to obtain $\vvv_x$. 
If $\vvv_x \neq \vvv_y$ then $\x = \sigma_1\dots \sigma_{q-1}\widehat{\sigma}_q \sigma_{q+1}\dots \sigma_\ell$,
 whereas $\sigma_1\dots \sigma_{p-1}\widehat{\sigma}_p \sigma_{p+1}\dots \sigma_\ell$ is a non-reduced expression. 
It  follows from Humphrey's Deletion condition \cite[Section 5.8]{hum}
that there exists   
some   $q'<p<q$ and that 
the subexpression $\vvv'=
\sigma_1\dots
\sigma_{q'-1}\widehat{\sigma}_{q'} \sigma_{q'+1}\dots
 \sigma_{p-1}\widehat{\sigma}_p \sigma_{p+1}\dots \sigma_\ell$
  is   also a reduced expression for $z$.
By applying a sequence of braids we may assume that $q'=p-1$ and $q=p+1$
(so that the cancellation in the non-reduced expression involves two adjacent reflections); thus   we can apply the (local) relation 
  \begin{equation}\label{needed}
     \begin{minipage}{1cm}
 \begin{tikzpicture}[scale=1]

\draw[densely dotted, rounded corners] (-0.25,0) rectangle (1.25-0.5,2);

\begin{scope}

\clip(-0.25,0) rectangle (1.25-0.5,2);

 \draw[magenta ,line width=0.08cm](0,0)    -- (0,2);

 \draw[magenta,fill=magenta ,line width=0.08cm](0.5,0)    -- (0.5,0.8) circle (3pt);

  \end{scope}

 \end{tikzpicture}
 \end{minipage}
=      \begin{minipage}{1cm}
 \begin{tikzpicture}[scale=1]

\draw[densely dotted, rounded corners] (-0.25,0) rectangle (1.25-0.5,2);

\begin{scope}

\clip(-0.25,0) rectangle (1.25-0.5,2);

 \draw[magenta ,line width=0.08cm,fill=magenta](0,0)    --++(90:0.8) circle (3pt
);

 \draw[magenta ,line width=0.08cm](0.5,2)    to [out=-90,in=90] (0.5,0);
;


  \end{scope}

 \end{tikzpicture}
 \end{minipage}
 +
     \begin{minipage}{1cm}
 \begin{tikzpicture}[xscale=1,yscale=-1]

\draw[densely dotted, rounded corners] (-0.25,0) rectangle (1.25-0.5,2);

\begin{scope}

\clip(-0.25,0) rectangle (1.25-0.5,2);

 \draw[magenta ,line width=0.08cm,fill=magenta](0.5,0)    --++(90:0.6) circle (3pt);

 \draw[magenta ,line width=0.08cm](0,2)    to [out=-90,in=180] (0.5,1.6);
;

 \draw[magenta,fill=magenta ,line width=0.08cm](0.5,2)    -- (0.5,1.2) 
 circle(3pt);

  \end{scope}

 \end{tikzpicture}
 \end{minipage}
  -
     \begin{minipage}{1cm}
 \begin{tikzpicture}[xscale=1,yscale=1]

\draw[densely dotted, rounded corners] (-0.25,0) rectangle (1.25-0.5,2);

\begin{scope}

\clip(-0.25,0) rectangle (1.25-0.5,2);

  \draw[magenta ,line width=0.08cm,fill=magenta](0,1)    --++(90:0.3) circle (3pt);
  \draw[magenta ,line width=0.08cm,fill=magenta](0,1)    --++(-90:0.3) circle (3pt);


 \draw[magenta ,line width=0.08cm](0,0)    to [out=90,in=180] (0.5,0.6);

 \draw[magenta ,line width=0.08cm](0.5,0)  --++(90:1.2)  to [out=90,in=-90] (0.5,2);

  \end{scope}

 \end{tikzpicture}
 \end{minipage}
 \end{equation}
to see that
\begin{multline*}
({\sf1}_ {\sigma_1\cdots \sigma_{p-1}} \otimes 
 {\sf spot}_{ {\sigma}_{p}} ^\emptyset 
   \otimes {\sf 1}_{ \sigma_{p+1} \cdots \sigma_{q-1}} \otimes {\sf spot}_{\sigma_q}^\emptyset \otimes {\sf 1}_{\sigma_{q+1} \dotsm \sigma_{\ell}}   )\\
=
({\sf1}_ {\sigma_1\cdots \sigma_{q'-1}} \otimes 
 {\sf spot}_{ {\sigma}_{q'}} ^\emptyset 
   \otimes {\sf 1}_{ \sigma_{q'+1} \cdots \sigma_{p-1}} \otimes {\sf spot}_{\sigma_p}^\emptyset \otimes {\sf 1}_{\sigma_{p+1} \dotsm \sigma_{\ell}}   ) 
   \pmod{\mathcal{H}_{W}^{<z}} 
\end{multline*}
(Any diagram with less than $\ell(z)$ propagating strands is zero; thus   the second term in \cref{needed} is immediately zero and 
 the third term is zero by the barbell  and   cyclotomic relations.)
 Now we consider the expression obtained from $\w$ by deleting $\sigma_{q'}$. 
If it is reduced, then by uniqueness of maximal degree basis elements, the expression obtained by deleting $\sigma_{q'}$ and $\sigma_p$ must be $\vvv_y$, and we are done. 
Otherwise, we can repeat the above   process.  
This must eventually terminate, as the indices of the deleted generators get smaller with each step at which point, we deduce that 
\begin{multline*}
({\sf1}_ {\sigma_1\cdots \sigma_{p-1}} \otimes 
 {\sf spot}_{ {\sigma}_{p}} ^\emptyset 
   \otimes {\sf 1}_{ \sigma_{p+1} \cdots \sigma_{q-1}} \otimes {\sf spot}_{\sigma_q}^\emptyset \otimes {\sf 1}_{\sigma_{q+1} \dotsm \sigma_{\ell}}   )\\
=
({\sf1}_ {\sigma_1\cdots \sigma_{t-1}} \otimes 
 {\sf spot}_{ {\sigma}_{t}} ^\emptyset 
   \otimes {\sf 1}_{ \sigma_{t+1} \cdots \sigma_{u-1}} \otimes {\sf spot}_{\sigma_u}^\emptyset \otimes {\sf 1}_{\sigma_{u+1} \dots   \sigma_{\ell}}   )
   \pmod{\mathcal{H}_{W}^{<z}}
   \end{multline*}
and the result follows, as we are working in the standard module 
$\Delta(z)={\mathcal{H}_{W}^{\leq z}}/ {\mathcal{H}_{W}^{<z}}$.  
\end{proof}

 \color{black}
\begin{cor}\label{compllllll}
   We have that 
$ 
{\rm Im}({\delta}_{\ell+1}) \subseteq {\rm ker}({\delta}_{\ell}) 
$,  
in other words 
  $ C_\bullet(1_{  W})$  
is a complex. 

\end{cor}

\begin{proof}\color{black}
We have defined the  differential (via the scalars $\varepsilon(w,z)$)  so that 
the composition  $\delta_\ell \delta_{\ell-1}$  restricted to a given
diamond  is equal to 
$
\varphi^{ w}_{ x} \varphi^{ x}_ {z} -
\varphi^{ w}_{y} \varphi_{z}^{y} 
$ and  so the result follows from \cref{lemmmmmmer}.  
\end{proof}

Now, we  apply the quotient functor 
 $:\mathcal{H}_{  W}{\rm -mod }\to  \mathcal{H}_{P\backslash W}{\rm -mod}$ 
to $C_\bullet(1_{  W})$    and hence obtain 
  $ C_\bullet(1_{P \backslash W}) = 
  \bigoplus_{\ell\geq 0 }{  \Delta}_\ell \langle \ell \rangle 
 $ 
together with the differential $(\delta_\ell)_{\ell\geq 0}$.  (We have abused notation slightly by identifying the 
differentials for   $ \mathcal{H}_{P\backslash W}$- and 
  $\mathcal{H}_{  W} $-modules.)

\begin{prop}\label{compllllll2}
We have that   $ C_\bullet(1_{  P\backslash W})$  
is a complex. 
\end{prop}
\begin{proof}
\color{black} For arbitrary $P\leq W$, we note that $\mathcal{H}_{P\backslash W}$  is the quotient of 
$\mathcal{H}_{W}$ by the parabolic annihilation relation (\ref{annihilation}).  Taking quotients preserves complexes and so the result follows from \cref{compllllll}.  
\end{proof}

\begin{rmk}
\color{black}  
In the quotient,  diamonds can ``collapse".  For example, if $y\not \in {^P}W$ then we obtain 
\begin{align}\begin{minipage}{5.5cm}
\begin{tikzpicture}
 \path(0,0)--++(90:30pt) coordinate(N);
 \path(0,0)--++(-90:30pt) coordinate(S);
 \path(0,0)--++(0:70pt) coordinate(E);
  \path(0,0)--++(180:70pt) coordinate(W);
\draw(N) node {$\Delta(x)$};
\draw(S) node {$0$};
\draw(W) node {$\Delta(w)$};
\draw(E) node {$\Delta(z)$};
\draw[->](173:70pt)to node[midway,above]  {$\varphi^ w_{ x}$} (120:30pt)  ; 
\draw[->](-173:70pt)to 
 (-120:30pt)  ; 
\draw[<-](7:70pt)to node[midway,above]  {$\varphi^x_{z}$} (60:30pt)  ; 
\draw[<-](-7:70pt)to 
  (-60:30pt)  ;   \end{tikzpicture}
\end{minipage}\end{align}
 in which case, we have that 
 $\varphi^{w}_{x} \varphi^{x}_ {\vvv}= 0$.  
 \color{black} 
 (To see this, simply note that the equality $\varphi^{w}_{x} \varphi^{x}_ {z}= 
\varphi^{w}_{y} \varphi_{z}^{y} 
$ continues to hold, but that the righthand-side of the equality factors through the zero module).
For example, this happens in \cref{heresmyexample}
with $w=\color{magenta}s_1 \color{cyan}s_2  $,  
$x=\color{magenta}s_1$, 
$y= \color{cyan}s_2$, and $z= 1_W$.  
\end{rmk}

We have already encountered one drawback of the $\ctau$-restriction functors from the previous section: they kill any standard   module $\Delta_{\leq \w}(x)$ such that $x<x\ctau \not\in {^PW}$  (and therefore the simple head is also killed).    
To remedy this, we define  slightly larger algebras 
\begin{align}
{\sf 1}_{\leq \ell}\mathcal{H}_{P\backslash W}^\Bbbk{\sf 1}_{\leq \ell}
\quad \text{for} \quad 
\color{black}
{\sf 1}_{\leq \ell}=\sum_{
\begin{subarray}c \w \in {\sf exp}_P^k(w) \\ 0\leq k \leq \ell \\ w\in W
\end{subarray}}{\sf 1}_{  \w}
\end{align}
and we define $\Res^{\ell+1}_{\ell} : 
{\sf 1}_{\leq \ell+1}\mathcal{H}_{P\backslash W}^\Bbbk{\sf 1}_{\leq \ell+1}
\to 
{\sf 1}_{\leq \ell}\mathcal{H}_{P\backslash W}^\Bbbk{\sf 1}_{\leq \ell}$ to be the functor 
 \begin{align} 
\Res^{\ell+1}_\ell= \bigoplus_{
  \begin{subarray}c
     \w \in {\sf exp}_P^\ell 
   \\ 
     \ctau  \in S
  \end{subarray}} \Res_{\leq \w}^{\leq \w\ctau}.
  \end{align}

\begin{lem}\label{importantforproof}
Let  $  1_{P\backslash W} \neq x,w \in {^PW}$ and suppose that $x\leq w$.  
We have that  $$\Res_{\leq \w}^{\leq    \w \ctau}(L_{\leq \w \ctau}(x))\neq 0$$ for some $\ctau \in S$ and $\w\in {\sf exp}_P(w)$.  
 Therefore 
 $\Res_{  \ell }^ {  \ell+1}(  L (x){\sf 1}_{
\leq \ell+1})= 0$ implies  $ x=1_{P\backslash W}$.  
\end{lem}
\begin{proof}
For $  1_{P\backslash W}\neq x \in  {^PW}$, there exists some $\ctau \in S$, $x' \in {^PW}$ such that 
$    x\ctau   = x' <x$. 
In which case,   $\x = \x' \ctau$ is a reduced expression for $x$ and ${\sf 1}_{\x'\ctau}\in L(x)$.  
Our assumption that $x \ctau=x' \leq x\leq w\leq w \ctau$ implies that  the preimage of 
 $ 
 {\sf 1}_{\x'\ctau}  
\in {\sf 1}_{\leq \w\ctau} \mathcal{H}_{P\backslash W}{\sf 1}_{\leq \w\ctau}$
under the map of (\ref{preimage}) is equal to $ 
   {\sf 1}_{\x' }  
\in {{\sf 1}_{\leq \w}\mathcal{H}_{P \backslash W}{\sf 1}_{\leq \w}}$
and so the result follows. 
  \end{proof}

We are now ready to prove that  $ C_\bullet(1_{P \backslash W}) $ is a   BGG resolution of the $\mathcal{H}_{P\backslash W}^\Bbbk$-module $L(1_{P\backslash W})$.  
 For $W$ the affine symmetric group, $P$ the maximal finite parabolic and $\Bbbk=\CC$, the existence of these BGG resolutions was conjectured by 
  Berkesch--Griffeth--Sam in \cite{conjecture}.  
  This conjecture was proven  by way of the {\sf KZ}-functor  in the context of  the  quiver Hecke algebras of  type $A$ (by the first and third authors with Jos\'e Simental,  \cite{BNS}).  
 In type $A$, the diagrammatic Hecke categories and (truncations of) quiver Hecke algebras were recently shown to be isomorphic in \cite{cell4us2}.  
 Thus  the following theorem generalises the BGG resolutions   \cite{conjecture,BNS}  
  to all Coxeter groups, $W$, and all parabolic subgroups, $P$, and arbitrary fields, $\Bbbk$.

\begin{thm}\label{maintheomre}
Fix 
   $W$ an arbitrary Coxeter group and fix $P$ an arbitrary parabolic subgroup.   The $\mathcal{H}_{P\backslash W}^\Bbbk$-complex $C_\bullet(1_{P\backslash W})$ is exact except in degree zero, where  $H_0(C_\bullet(1_{P\backslash W}))=L(1_{P\backslash W}).$  Moreover, we have that  $$ L(1_{P\backslash W})=\Bbbk
\{c_{\sts}  
  \mid 
   \Shape_k(\sts)={\sf 1}_{P\backslash W} \text { for all }k\geq 1 )   \}.  
$$
     \end{thm} 

\begin{proof} 
By applying the restriction functor to  \cref{compllllll2},  we have that 
\begin{align}
 \Res^{ {  \ell+1}}_{ {  \ell} }
 ( C_\bullet(1_{P\backslash W}){\sf 1}_{\leq \ell+1})
 \end{align}
forms a complex of ${\sf 1}_{\leq \ell}\mathcal{H}_{P\backslash W}^\Bbbk{\sf 1}_{\leq \ell}
$-modules.  Moreover, we can idempotent-truncate  
\begin{equation}\label{ppppp}
D^{\ctau,\w}_\bullet (1_{P\backslash W}) =  ( \Res^{ {  \ell+1}}_{ {  \ell} }({\sf 1}_{\leq \ell+1}C_\bullet(1_{P\backslash W}))) ( {\sf 1}_{\leq \w}\otimes {\sf 1}_{\ctau})
 \end{equation}and hence obtain a  complex of 
 ${\sf1}_{\leq \w}{ \mathcal{H}_{P\backslash W}^\Bbbk{\sf1}_{\leq \w}
 }$-modules (through the identification of $
  {\sf 1}_{\leq \w } \mathcal{H}_{P\backslash W}{\sf 1}_{\leq \w }
  \hookrightarrow 
  {\sf 1}_{\leq \w\ctau} \mathcal{H}_{P\backslash W}{\sf 1}_{\leq \w\ctau}
  $).   
Let $x,y\in {^PW}$ with $x=  y\ctau >y$.     
For   $y \in {^PW}$, we have that already seen that 
\begin{align}
   0\to \{c_{\stt^+}\mid \stt \in \Std_{\leq \w}(y)\} \to 
 \Res^{ {\leq \w \ctau}  }_{  {\leq \w  } }(   \Delta_{\leq \w\ctau}(x))
    \to 
     \{c_{\stt^+}\mid \stt \in \Std_{\leq \w}(x)\} \to 0
\end{align}
   and 
\begin{align}
   0\to \{c_{\stt^-}\mid \stt \in \Std_{\leq \w}(y)\} \to 
\Res^{ {\leq \w \ctau}  }_{  {\leq \w  } }(    \Delta_{\leq \w\ctau}(y)) 
    \to 
     \{c_{\stt^-}\mid \stt \in \Std_{\leq \w}(x)\} \to 0
\end{align}
 where in both cases the submodule is isomorphic to $\Delta_{\leq \w}(y)$  and the quotient module is    
    isomorphic to $\Delta_{\leq \w}(x)$.  
Since $\x=\y\ctau$,  we have  that 
\begin{align}
\varphi^ x_ y(c_{\stt^+})=({\sf1}_ {\y} \otimes 
 {\sf spot}^{ \ctau} _\emptyset 
 )   c_{\stt^+}
   =   c_{\stt^-}
\end{align}
for any $\stt \in \Std_{\leq \w}(x)$ or  $\stt\in \Std_{\leq \w}(y)$    by definition.   
Therefore, we have that 
\begin{align}
\big( \Res^{ {\leq \w\ctau}}_{  {\leq \w}}
 \circ  
\varphi^{ x
}_{ y
} 
\big)
=   {\sf id}_\x \langle 1 \rangle  +  
  {\sf id}_\y \langle 1 \rangle   
  \end{align}
for $x,y\in {^PW}$ where 
\begin{align}
 {\sf id}_\vvv \langle 1 \rangle  \in {\rm Hom}_{{\sf1}_{\leq \w}{ \mathcal{H}_{P\backslash W}^\Bbbk{\sf1}_{\leq \w}
 }}(\Delta_{\leq \w}(z)\langle \ell(z)\rangle,\Delta_{\leq \w}(z)\langle \ell(z)+1 \rangle)
\end{align}
is simply the graded shift of the identity map for $\vvv=\x,\y $ for $x,y \in W$.  
This implies that
\begin{align}
D^{\ctau,\w}_\bullet (1_{P\backslash W}) =  \bigoplus_{
\begin{subarray}c y\ctau > y
\\ 
\y \leq  \w
\end{subarray} }\big(\Delta(y)\langle \ell(y)\rangle\oplus 
\Delta(y)\langle \ell(y)+1\rangle\big)
\end{align}
with differential 
\begin{align}
  \Res^{ {\leq \w \ctau}  }_{  {\leq \w  } }
\circ  \delta_\ell =
 \sum_ 
 {\begin{subarray}c
 (x,y) \in  {\sf CP}_\ell\\
 \x = \y \ctau 
\end{subarray} } ({\sf id}_\x\langle 1  \rangle 
 +
  {\sf id}_\y \langle 1 \rangle ) 
  + 
 \sum_ 
 {\begin{subarray}c
 (s,t) \in  {\sf CP}_\ell\\
 {\underline{s}} \neq  {\underline{t}} \ctau 
\end{subarray} } (\Res^{ {\leq \w\ctau}}_{  {\leq \w}}
 \circ  
\varphi^{{{s}}
}_{{{t}}
} ).
\end{align}
 Thus we have that 
\begin{align}H_j(
(  \Res^{  {\leq \w\ctau}}_{  {\leq \w}
 }( (
 C_\bullet(1_{P\backslash W})
 ) ( {\sf 1}_{\leq \w}\otimes {\sf 1}_{\ctau})
 ){\sf 1}_{\leq \ell+1})=0
 \end{align} for all $j\geq 0$. 
 Now, summing over all $\ctau \in S$, $w\in W$, and $\w \in {\sf  exp}_P^\ell(w)$ we deduce that 
\begin{align}
\Res^{ { \ell+1}}_{  {  \ell}
 }( C_\bullet(1_{P\backslash W}){\sf 1}_{\leq \ell+1})
 \end{align}
  forms  a complex with zero homology in every degree. 
By \cref{importantforproof},  we have that restriction kills no simple $\mathcal{H}_{P\backslash W}^\Bbbk$-module $L(w)$ for $1\neq w\in {^P}W$.   
 Moreover,  
\begin{align}
{\rm Head}( \Delta (
 1_{P\backslash W}
 ){\sf 1}_{\leq \ell+1})= 
 L (1_{P\backslash W}) {\sf 1}_{\leq \ell+1} \not \subset  {\rm Im}(\delta_1)
\end{align}
 and 
 $[{\sf 1}_{\ell+1}\Delta (w):{\sf 1}_{\ell+1} L (1_{P\backslash W})]=0$ for $  1_{P\backslash W}\neq w \in {^P}W$
 simply because the highest weight structure on $\mathcal{H}
 _{P \backslash W}$ is given by the Bruhat order.  
Therefore 
\begin{align}
H_j( C_\bullet(1_{P\backslash W}){\sf 1}_{\leq \ell+1})  
 =
 \begin{cases}
  L (1_{P\backslash W}){\sf1}_{\leq \ell+1} &\text{if }j=0\\
 0		&\text{otherwise}.
 \end{cases}
 \end{align}
 Finally, we have proven that $L(1_{P\backslash W})$ is killed by multiplication by the 
idempotent ${\sf 1}_\ctau$ at the $\ell$th point for any $\ell\geq 1$ and for any $\ctau \in S.$  
Thus $L(1_{P\backslash W})$ is spanned by $c_\sts$ for $\sts$ the empty tableau, as required.   
   \end{proof}

We immediately deduce the following corollary, which   is new even for the classical (inverse) parabolic and non-parabolic Kazhdan--Lusztig polynomials (in other words, for $\Bbbk$ the complex field).  Indeed, this seems to be the first non-trivial   family of parabolic   ($p$-)Kazhdan--Lusztig polynomials which admits a   uniform construction across all Coxeter groups and all parabolic subgroups.  

\begin{cor}[The Weyl--Kac character formula for Coxeter groups]
In the graded Grothendieck group of $\mathcal{H}_{{P \backslash W}}$, we have that 
$$
[L(1_{{P \backslash W}})]= \sum_{w\in {^PW}}(-\grade)^{\ell(w)}[\Delta(w)] 
$$
Thus the complete  first row of the inverse $p$-Kazhdan--Lusztig matrix is given by 
 $$
{^p}n_{1,w }^{-1}=(-\grade)^{\ell(w)} 
$$
for all  $w\in {^PW}$.  
\end{cor}

\begin{thm}
The   module $L(1_{P\backslash W})$ is both the unique homogeneous   $\mathcal{H}_{P\backslash W}^\Bbbk$-module and  the unique 1-dimensional $\mathcal{H}_{P\backslash W}^\Bbbk$-module.  
\end{thm}

\begin{proof}
That the module $L(1_{P\backslash W})$ is homogeneous is clear (as it is 1-dimensional).  We now prove the converse, namely  for any $1\neq w\in {^PW} $ we show that 
$L(w)$ is inhomogeneous and of dimension strictly greater than 1.  
 Let $1\neq w\in  {^PW}$ and choose $\tau$ such  that $w\ctau=w' <w$.  
 By \cref{LEW1}, the elements 
 \begin{equation}\label{hiydddd}
{\sf 1}_{\w' } \otimes {\sf spot}_\ctau^\emptyset \otimes {\sf 1}_\ctau
\qquad
{\sf 1}_{\w' } \otimes  {\sf fork}_{\ctau\ctau}^{\ctau}
\end{equation}span $\Delta_{\leq \w \ctau\ctau}(w)$.  The former is homogeneous of degree $+1$ and the latter is homogeneous of degree $-1$.  We have that 
\begin{align}
({\sf 1}_{\w' } \otimes  {\sf fork}_{\ctau\ctau}^{\ctau}) 
 ({\sf 1}_{\w' } \otimes  {\sf spot}^\ctau_\emptyset \otimes {\sf 1}_\ctau) 
={\sf 1}_{\w'}
\end{align}
  and 
\begin{align}
({\sf 1}_{\w'} \otimes  {\sf spot}_\ctau^\emptyset \otimes {\sf 1}_\ctau )
({\sf 1}_{\w'} \otimes  {\sf spot}^\ctau_\emptyset \otimes {\sf 1}_\ctau )  =
 {\sf 1}_{\w'} \otimes   
{\sf spot}_\ctau^\emptyset{\sf spot}^\ctau_\emptyset \otimes {\sf 1}_\ctau  
 =0
 \pmod{ \mathcal{H}_{P\backslash W}^{< \w}{\sf 1}_{\leq \w\ctau\ctau}}  
\end{align}
since the  degree of this element  is $+2$ (whereas the degree of the idempotent spanning the weight space of the cell module is, of course, of degree 0) and 
$$
({\sf 1}_{\w'} \otimes  {\sf fork}_{\ctau\ctau}^{\ctau })
({\sf 1}_{\w'} \otimes  {\sf fork}^{\ctau\ctau}_{\ctau\emptyset}) 
 =0 .  
$$ 
Therefore the Gram matrix for this weight space of the cellular form is 
the $2\times 2$-matrix with 0s on the diagonal and 1s off the diagonal.  This matrix has rank 2 and so neither element in (\ref{hiydddd}) belongs to the radical of the cellular form.  
Therefore both  elements in (\ref{hiydddd}) belong to $L_{\leq \w\ctau\ctau}(w)$ 
 and   the result follows.    
\end{proof}

   \begin{rmk}
We recall from the introduction  that the conjecture of Berkesch--Griffeth--Sam (or rather, its equivalent formulation for homogeneous representations of quiver Hecke algebras)  follows immediately  from Theorem B.  This might be surprising to the reader familiar with the homogeneous representations of quiver Hecke algebras.    
   In \cite{JEMS} it is shown that there are up to  $e-1$ distinct homogeneous representations of any block of the quiver Hecke algebras (and for sufficiently large rank, there are precisely $e-1$ such representations for a ``regular block").  Whereas, in this paper we have seen that there is precisely one homogeneous representation of $\mathcal{H}_{P\backslash W}^\Bbbk$ for $\mathfrak{S}_{h}=P\subset W=\widehat{\mathfrak{S}}_{h} $ for  $h \in \NN$.  
Therefore, one might think that there are  ``more" homogeneous representations of the quiver Hecke algebra.  
 However, for each $1\leq h <e$ there is an isomorphism between a finite truncation of $\mathcal{H}_{P\backslash W}^\Bbbk$ and  the Serre quotient of the quiver Hecke algebra corresponding to the set of partitions with at most $h$ columns  \cite{cell4us2}. 
Through these isomorphisms, one can obtain the
 $e-1$ distinct  BGG resolutions of the  $e-1$ distinct homogeneous simple modules of the quiver Hecke algebra predicted by  Berkesch--Griffeth--Sam  \cite{conjecture}.  
   \end{rmk}



%
%

%

We now provide an elementary infinite family of simple modules which do not admit BGG resolutions, in order to justify our claim in the introduction  that such resolutions are ``rare".  
In \cite{bgg} an example of such a simple for  $W=\mathfrak{S}_4$ is given.  
We focus on  the simplest case, namely that of the 
anti-spherical category  controlling  the algebraic group $ {\rm SL}_2(\Bbbk)$.

\color{black}
\begin{prop}\label{potrpotrptor} 
 Let $\Bbbk$ be a field of   finite characteristic $p> 0$.   
There exist  infinitely many   simple $\mathcal{H}_{  \mathfrak{S}_2\backslash \widehat{\mathfrak{S}}_2}^\Bbbk$-modules which do  not admit   BGG resolutions.  
\end{prop}

\begin{proof}\color{black}
This is a standard Temperley--Lieb type result, we provide references in \cref{belooow}  below but we include a proof for the sake of completeness.  
 The Coxeter presentation of $\widehat{\mathfrak{S}}_2$ is   $   \langle \csigma, \ctau  \mid \ctau ^2=\csigma^2=1\rangle$ and we let $P$ denote the finite parabolic generated by the reflection $\ctau  $.    
We will provide an infinite family of examples of 
$x\in W$ such that  $\rad(\Delta(x))$ is not generated by the homomorphic images of standard modules,  thus showing that each such  $L(x)$ does not admit a BGG resolution.  
For $n\in \NN$, we set \begin{equation}
\x= \underbrace{\csigma\ctau\csigma\ctau\csigma \dots   }_{{np}-1}\qquad 
 \vvv= 
 \underbrace{\csigma\ctau\csigma\ctau\csigma \dots   }
 _{{np}}\qquad
\y= 
\underbrace{\csigma\ctau\csigma\ctau\csigma \dots   }
_{{np}+1}  
\end{equation}  so that, in the   notation of \cref{itemmmmmmmmm}, we have 
 \begin{equation}
{\sf 1}_\x= 
 {\sf 1}_{\csigma\ctau}^{{{np}-1}}
\qquad 
{\sf 1}_ \vvv= 
 {\sf 1}_{\csigma\ctau}^{{{np} }}
\qquad
{\sf 1}_\y= 
 {\sf 1}_{\csigma\ctau}^ {{np}+1}  .
\end{equation} 
Suppose that  $L(y)$ is a subquotient of  $\Delta(x)$ 
and that $L(y)$ belongs to the submodule  generated by the homomorphic images of standard modules.  
Then $L(y)$ must be in the image of a homomorphism from $\Delta(y)$ or $\Delta(z)$ 
by highest weight theory.  (In more detail, we simply note that $[\Delta(w):L(y)]\neq 0$ implies that $w\leq y$ and that 
$ \Hom_{\mathcal{H}_{P\backslash W}^\Bbbk}(\Delta(w), \Delta(x))\neq 0$ implies that $x\leq w$ and therefore 
$x\leq w \leq y$; thus $w\in \{x,y,z\}$.)  
%
The module $\Delta ( x){\sf 1}_{\underline{y}}$ is $({np}-1)$-dimensional and spanned by the light leaves basis elements 
\begin{equation}
f_k:=  
\begin{cases}
({\sf1}_{\csigma\ctau}^{k-1} \otimes {\sf fork}_{\csigma \csigma}^{\csigma }  \otimes {\sf1}_{\ctau\csigma}^{np-1-k} ) 
  ( {\sf1}_{\csigma\ctau}^k  \otimes {\sf spot}^\emptyset_{\ctau}
 \otimes  {\sf1}^{np-k}_{\csigma\ctau}  )  & \text{if $k$ is odd,}
\\ 
({\sf1}_{\csigma\ctau}^{k-1} \otimes {\sf fork}_{\ctau \ctau}^{\ctau }  \otimes {\sf1}_{\csigma\ctau}^{np-1-k} ) 
  ( {\sf1}_{\csigma\ctau}^k  \otimes {\sf spot}^\emptyset_{\csigma}
 \otimes  {\sf1}^{np-k}_{ \ctau\csigma}  ) & \text{if $k$ is even,}
 \end{cases}
\end{equation}
 for $1\leq k < {np}$.
We now calculate the full submodule structures of standard modules and hence verify that $L(y)$ is not in the image of any homomorphism  $ \Hom_{\mathcal{H}_{P\backslash W}^\Bbbk}(\Delta(w), \Delta(x))$ for $w \in \{x,y,z\}$.

\smallskip\noindent{\bf Decomposition numbers.} We will prove that the Gram matrix of $\Delta (x){\sf 1}_{\underline{y}}$ has rank $\dim (\Delta_{\leq \y}(x))-\dim (L_{\leq \y}(y))=np-1-1=np-2$.  Thus proving that 
$[\Delta(x):L(y)]=1 $ using 
\cref{fuckoffre} (this is a standard cellular argument).  
 The Gram-matrix of the cell-form of $\Delta ( x){\sf 1}_{\underline{y}}$ has $-2$ for  each of the diagonal entries and $1$ for each of the super and subdiagonals (in other words, it is equal to  $-1$ times the Cartan matrix of type $A_{np-1}$).  The determinant of this matrix is $np$ which is zero in $\Bbbk$.  Over $\Bbbk$, the rank of this matrix is $np-2$  and so 
$[\Delta(x): L(y)  ]  = 1$ as required.

\smallskip\noindent{\bf Submodule structures and homomorphisms. }
By considering the light leaves basis, one deduces  that $\Delta_{\leq \y}(y))=L_{\leq \y}(y)$ is 1-dimensional
 and that 
$\Delta_{\leq \y}(z))$ is 2-dimensional with simple head 
$L_{\leq \y}(z)$ and simple socle $L_{\leq \y}(y)$.  
Suppose that the socle of $\Delta_{\leq \y}(x))$ contains a submodule isomorphic to $L_{\leq \y}(y)$.
This submodule must be the span   of some element 
 $g=\sum \alpha_k f_k$ for $\alpha_k \in \Bbbk$ 
 with $gf_k^\ast$ vanishing in $\Delta_{\leq \y}(x)$ for all $1\leq k<np$, i.e.
\begin{equation*}
 g f_k^\ast \in \mathcal{H}_{  \mathfrak{S}_2\backslash \hat{\mathfrak{S}}_2}^{<x} 
\end{equation*}
 for all   $1\leq k <np$. 
 We have that $f_{k\pm1} f_k^\ast  ={\sf 1}_\x$ and 
 $f_{k }f_k^\ast=-2\cdot {\sf 1}_\x$ modulo $\mathcal{H}_{  \mathfrak{S}_2\backslash \widehat{\mathfrak{S}}_2}^{<x}$.  
Thus
$$g=\alpha \big( f_{np-1}	+2   f_{np-2} +3f_{np-3}+4f_{np-4}
+5f_{np-5}+\dots + (np-1) f_{1}\big),$$
for some $\alpha \in \Bbbk\setminus \{0\}$.  
However, we notice that for $g$ as above,
\begin{align*}
\alpha {\sf 1}_{\csigma\ctau}^{np-1} \otimes {\sf spot}^\emptyset_{\ctau} & = g({\sf 1}_{\csigma\ctau}^{np-2} \otimes	{\sf spot}^\csigma_\emptyset) & & \text{if $np$ is even,} \\
\alpha {\sf 1}_{\csigma\ctau}^{np-1} \otimes {\sf spot}^\emptyset_{\csigma} & =g({\sf 1}_{\csigma\ctau}^{np-2} \otimes	{\sf spot}^\ctau_\emptyset) & & \text{if $np$ is odd,}
\end{align*}
and so the submodule generated by $g$ contains $L_{\leq \y}(z)$ as a submodule. Thus  $\Delta_{\leq \y}(x)$ is uniserial with simple head 
$L_{\leq \y}(x)$, simple socle 
$L_{\leq \y}(z)$, and the middle composition factor $L_{\leq y}(y)$. 
Thus $L_{\leq \y}(x)$  
 not  the   image of a homomorphism from $\Delta_{\leq \y}(z)$ or 
$\Delta_{\leq \y}(y)$, as required.   
  \end{proof}
  
%

\begin{rmk}
There is a theory of  ${``\rm SL}_2(\Bbbk)$-strings{\rm "} for algebraic groups. 
This theory allows to inflate  
decomposition numbers and extension groups  between standard  modules   for ${\rm SL}_2(\Bbbk)$ to 
calculate certain  decomposition numbers and homomorphisms between standard modules for  arbitrary algebraic groups (see \cite[Part II, 5.21 (2)]{Jan03} and \cite[Introduction]{MR1324539} for decomposition numbers and extension groups respectively).  
One can use the equivalence between algebraic groups and anti-spherical Hecke categories from \cite{BR} in order  to translate these statements to the setting of  $\mathcal{H}_{P\backslash W}$-modules (in the case that 
 $P$ is the maximal parabolic of an affine Weyl group $W$).  
Thus one can use \cref{potrpotrptor} to provide many examples of simple $\mathcal{H}_{P\backslash W}$-modules which do not admit BGG resolutions. We do not go into further details here.  \end{rmk}

 \color{black}
\begin{rmk}\label{belooow}
Through the isomorphism of \cite{cell4us2}, we can  rephrase the above as a question concerning  decomposition numbers and homomorphisms for the symmetric group
$\mathfrak{S}_{np^2+p}$ in characteristic $p>2$.    
We let 
$S(\lambda)$ denote  the Specht module and $D(\mu)$ is the simple head for $\lambda,\mu$ partitions (the latter $p$-regular).  
We have that $
 [S( np^2,p): D(np^2+p)  ] = 1,
$ however 
$ D({np^2+p} )$ is not in the image of any homomorphism  $
 \Hom_{\Bbbk\mathfrak{S}_{np^2+p}}(S(\lambda) ,S(np^2,p) ))$.  
This example was already known to Gordon James in \cite[24.5 Examples]{James} for $p=2$ and the general case is similar, see \cite[24.4 Theorem]{James} and \cite[24.15 Theorem]{James}.    
For $p$ arbitrary, the full submodule structures of Specht modules labelled by 2-part partitions were determined in \cite{MR2025550}.  
\end{rmk}


%

\begin{Acknowledgements*}
We would like to thank George Lusztig and Stephen Donkin for their helpful comments.  
We would also like to thank the anonymous referees for their detailed reading of the paper and their helpful suggestions.  
The   authors are grateful for funding from EPSRC  grant EP/V00090X/1, the Royal Commission for the Exhibition of 1851, and European Research Council   grant No.~677147, respectively.

\end{Acknowledgements*}
       \bibliographystyle{amsalpha}   
\bibliography{master}

 \end{document}